\documentclass{amsart}
\usepackage{amsmath,amssymb,latexsym}
\usepackage{amsthm}
\usepackage{amsfonts}
\usepackage{mathtools,caption}
\usepackage{autobreak}
\usepackage{bm}

\usepackage[OT2,OT1]{fontenc}

\usepackage{mleftright}
\usepackage{xparse}
\usepackage{mleftright}
\usepackage{bbm}
\usepackage{enumitem}

\usepackage{array}

\usepackage{xcolor}
\usepackage[colorlinks=true, linkcolor=blue, citecolor=blue]{hyperref}

\usepackage{aliascnt}
\usepackage{cleveref}

\numberwithin{equation}{section}

\mathtoolsset{showonlyrefs=true}

\newcommand{\dsum}{\displaystyle\sum}

\newcommand{\set}[1]{\left\{ #1 \right\}}
\newcommand{\Set}[2]{\left\{ #1 \mathrel{} \middle| \mathrel{} #2 \right\}}

\newcommand{\abs}[1]{\left \lvert #1 \right \rvert}
\newcommand{\leg}[2]{\left(\frac{#1}{#2}\right)}

\newcommand{\floor}[1]{\left\lfloor#1\right\rfloor} 
\newcommand{\inv}[1]{#1^{-1}}
\newcommand{\dprod}{\displaystyle\prod}

\makeatletter
\def\widebar{\accentset{{\cc@style\underline{\mskip10mu}}}}
\makeatother

\makeatletter
\def\wideubar{\underaccent{{\cc@style\underline{\mskip10mu}}}}
\makeatother

\def\x{^{\times}}
\def\inv{^{-1}}
\def\1{\mathbbm{1}}

\def\jac{\mathrm{Jac}}

\def\GL{\mathrm{GL}}

\newcommand{\ceil}[1]{\left\lceil#1\right\rceil} 

\def\A{\mathbb{A}}

\def\C{\mathbb{C}}

\def\F{\mathbb{F}}

\def\P{\mathbb{P}}
\def\Q{\mathbb{Q}}
\def\R{\mathbb{R}}

\def\Z{\mathbb{Z}}

\def\mf{\mathfrak}

\def\mc{\mathcal}

\DeclareMathOperator{\gal}{Gal}

\DeclareMathOperator{\ord}{ord}

\DeclareMathOperator{\Tr}{Tr}

\DeclareMathOperator{\frob}{Frob}

\theoremstyle{definition}

\newtheorem{theoremeigo}{Theorem}[section]

\newaliascnt{definitioneigo}{theoremeigo}
\newtheorem{definitioneigo}[definitioneigo]{Definition}
\aliascntresetthe{definitioneigo}

\newaliascnt{lemmaeigo}{theoremeigo}
\newtheorem{lemmaeigo}[lemmaeigo]{Lemma}
\aliascntresetthe{lemmaeigo}

\newaliascnt{propositioneigo}{theoremeigo}
\newtheorem{propositioneigo}[propositioneigo]{Proposition}
\aliascntresetthe{propositioneigo}

\newaliascnt{corollaryeigo}{theoremeigo}

\aliascntresetthe{corollaryeigo}

\newaliascnt{exampleeigo}{theoremeigo}
\newtheorem{exampleeigo}[exampleeigo]{Example}
\aliascntresetthe{exampleeigo}

\newaliascnt{remarkeigo}{theoremeigo}
\newtheorem{remarkeigo}[remarkeigo]{Remark}
\aliascntresetthe{remarkeigo}

\newaliascnt{noteeigo}{theoremeigo}
\newtheorem{noteeigo}[noteeigo]{Note}
\aliascntresetthe{noteeigo}

\newaliascnt{problemeigo}{theoremeigo}

\aliascntresetthe{problemeigo}

\newaliascnt{conjectureeigo}{theoremeigo}

\aliascntresetthe{conjectureeigo}

\newaliascnt{notationeigo}{theoremeigo}

\aliascntresetthe{notationeigo}

\crefname{theoremeigo}{theorem}{theorems}
\Crefname{theoremeigo}{Theorem}{Theorems}
\crefname{definitioneigo}{definition}{definitions}
\Crefname{definitioneigo}{Definition}{Definitions}
\crefname{lemmaeigo}{lemma}{lemmas}
\Crefname{lemmaeigo}{Lemma}{Lemmas}
\crefname{propositioneigo}{proposition}{propositions}
\Crefname{propositioneigo}{Proposition}{Propositions}
\crefname{corollaryeigo}{corollary}{corollaries}
\Crefname{corollaryeigo}{Corollary}{Corollaries}
\crefname{exampleeigo}{example}{examples}
\Crefname{exampleeigo}{Example}{Examples}
\crefname{remarkeigo}{remark}{remarks}
\Crefname{remarkeigo}{Remark}{Remarks}
\crefname{noteeigo}{note}{notes}
\Crefname{noteeigo}{Note}{Notes}
\crefname{problemeigo}{problem}{problems}
\Crefname{problemeigo}{Problem}{Problems}
\crefname{conjectureeigo}{conjecture}{conjectures}
\Crefname{conjectureeigo}{Conjecture}{Conjectures}
\crefname{notationeigo}{notation}{notations}
\Crefname{notationeigo}{Notation}{Notations}
\crefname{claimeigo}{claim}{claims}
\Crefname{claimeigo}{Claim}{Claims}

\theoremstyle{definition}

\makeatletter

\title{Root numbers for twisted Fermat quotient curves}
\author{Ryosuke Yanagihara
}

\address{Ryosuke Yanagihara\\
Mathematical inst. Tohoku Univ.\\
6-3, Aoba, Aramaki, Aoba-ku, Sendai, 980-8578, JAPAN}
\email{yanagihara.ryosuke.t1@dc.tohoku.ac.jp}



\allowdisplaybreaks[4]
\begin{document}

\begin{abstract}
	Let $\ell$ be an odd prime, $N \geq 1$ be an integer, and $\delta \geq 1$ be an $\ell^N$-th power free integer such that $\ord_{\ell}(\delta) = 0$ or $\ell \nmid \ord_{\ell}(\delta)$. In this paper, we give an explicit formula for the root number of the Hecke character associated with a certain quotient curve of the twisted Fermat curve $X^{\ell^N} + Y^{\ell^N} = \delta$. This result gives a generalization of Stoll \cite{stoll2002arithmetic} and Shu \cite{shu2021root}.
\end{abstract}

\maketitle
\tableofcontents

\section{Introduction}\label{Intro}
Let $\ell$ be an odd prime, $N \geq 1$ be an integer, and $\delta \geq 1$ be an integer that is $\ell^N$-th power free (i.e. for any prime $p$, $\delta$ is not divisible by $p^{\ell^N}$), satisfying either $\ord_{\ell}(\delta) = 0$ or $\ell \nmid \ord_{\ell}(\delta)$. The nonsingular projective curve defined over $\Q$ by $X^{N} + Y^{N} = \delta$ is called a \textit{twisted Fermat curve}, and is a widely studied object in arithmetic geometry (cf. \cite{gross1978some}, \cite{stoll2002arithmetic}, \cite{shu2021root}, \cite{diaconu2005twisted}). We consider when the degree is an odd prime power $\ell^N$.

Let $r, s, t > 0$ be integers such that $r + s + t = \ell^N$ and $\ell \nmid rst$. Let $1 \leq i \leq N$ be an integer. We define a family of curves $C_i^{(\delta; r_i, s_i, t_i)}$. In this introduction, let $\overline{a}$ denote the nonnegative remainder of an integer $a$ modulo $\ell^i$.

\begin{definitioneigo}
    First, we set the following positive integers $1\leq r_i, s_i,t_i<\ell^i$ as follows:
    $$(r_i,s_i,t_i):=
    \begin{dcases}
        (\overline{r}, \overline{s}, \overline{t})\ &(\overline{r}+\overline{s}+\overline{t}=\ell^i)\\
        (\ell^i-\overline{r}, \ell^i-\overline{s},\ell^i-\overline{t})\ &(\overline{r}+\overline{s}+\overline{t}=2\ell^i).
    \end{dcases}$$
Next, let $\delta_i:=\dprod_{p} p^{\overline{\ord_{p} (\delta)}}.$
Here, we note that $r_N=r,\ s_N=s,\ t_N=t,\ \delta_N=\delta$ and $\delta_i$ is $\ell^i$-th power free. In addition, $\ord_{\ell}(\delta_i) = 0$ or $\ell \nmid \ord_{\ell}(\delta_i)$.

Now, let  $F_{i}=F_{i,\delta_i}$ denote the nonsingular projective curve over $\Q$ defined by 
$$X^{\ell^i}+Y^{\ell^i} =\delta_i,$$ and we call it an \textit{$\ell^i$-th twisted Fermat curve}.

Let  $C_i = C_i^{(\delta; r_i, s_i, t_i)}$ denote the nonsingular projective curve over $\Q$ defined by 
$$y^{\ell^i} = x^{r_i}(\delta_i - x)^{s_i},$$ and we call it an \textit{$\ell^i$-th twisted Fermat quotient curve}.
\end{definitioneigo}
\begin{noteeigo}
    The $\ell^i$-th twisted Fermat quotient curve $C_i$ is a quotient of $F_{i}$ by the correspondence $F_{i}\to C_i,\ (X,Y)\mapsto (x,y)=(X^{\ell^i},X^{r_i}Y^{s_i})$ and it has genus $\frac{1}{2}(\ell^i-1)$.
    Notice that $C_i^{(\delta_i; r_i, s_i, t_i)}$ is isomorphic
to $C_i^{(\delta_i; \ell^i-r_i, \ell^i-s_i, \ell^i-t_i)}$ over $\Q$ by
the correspondence $(x,y)\longleftrightarrow
(x,x(\delta_i -x)/y)$ and it is consistent with the definition of $r_i,s_i$. 
\end{noteeigo}
Furthermore, $C_i$ is also a quotient curve over $\Q$ of $C_{i+1}$ when $i<N$. In fact, writing $W_i^{(\Delta;a,b)}$ for the smooth projective model of $y^{\ell^{i}}=x^a(\Delta-x)^b\ (a,b\in\Z,\ 0<a,b<\ell^i,\ \Delta\in\Q\x)$, there is a surjection
$$C_{i+1}
\to W_i^{(\delta_{i};\overline{r_{i+1}},\overline{s_{i+1}})}$$
locally given by $(x,y)\mapsto ({\lambda_i}^{-\ell^i}x,{\lambda_i}^{-\overline{r_{i+1}}-\overline{s_{i+1}}}\{x^{R_i}(\delta_{i+1}-x)^{S_i}\}\inv y^{\ell})$
when we write $r_{i+1}=R_i\ell^i+\overline{r_{i+1}},\ s_{i+1}=S_i\ell^i+\overline{s_{i+1}}$, and $\delta_{i+1}=\lambda_i^{\ell^i}\delta_i$.
Here, since that $W_i^{(\Delta;a,b)}\simeq_{\Q} W_i^{(\Delta;\ell^i-a,\ell^i-b)}$ holds by
the correspondence $(x,y)\longleftrightarrow
(x,x(\Delta -x)/y)$, and $\overline{r_{i+1}}\in\set{\overline{r},\ell^i-\overline{r}}=\set{r_i,\ \ell^i-r_i}$ and the same for $s_{i+1}$, the right term is isomorphic to $W_i^{(\delta; r_i, s_i, t_i)}=C_i$.


Thus, there is a quotient map $\pi_i:C_{i+1}\longrightarrow
 C_i$ defined over $\Q$. We define 
 $$\jac(C_i)^{\rm{new}}:=\jac(C_i)/\pi^\ast_{i-1}(\jac(C_{i-1}))$$
 for $i>1$ and $\jac(C_1)^{\rm{new}}:=\jac(C_1)$. The automorphism $\sigma_i:(x,y)\mapsto (x,\zeta_{\ell^i}y)$ on 
 $C_i$ yields the embedding $\Q(\zeta_{\ell^i})\hookrightarrow
 {\rm End}(\jac(C_i)^{\rm{new}})\otimes_\Z\Q,\ 
 \zeta_{\ell^i}\mapsto \sigma^\ast_i$. 
 Since $2{\rm dim}(\jac(C_i)^{\rm{new}})=[\Q(\zeta_{\ell^i}):\Q]$, the variety $\jac(C_i)^{\rm{new}}$ has CM by $\Q(\zeta_{\ell^i})$. 
The quotient curves $C_i$ ($1\le i\le N$) have rich arithmetic properties as shown in \cite{gross1978some} for the case when $N = \delta = 1$. 

The Jacobian variety $J_N$ of $C_N$ 
can be decomposed as follows:
$$ J_N \stackrel{\Q}{\sim} \bigoplus_{1 \leq i \leq N} \jac(C_i)^{\rm{new}}. $$
Then we have a relation on $L$-functions:
$$L(s,J_N)=\prod_{1\leq i\leq N} L(s,\jac(C_i)^{\rm{new}}).$$
On the other hand, for each factor  $\jac(C_i)^{\rm{new}}$, one can associate a 
unitary Hecke character  $\phi_{\delta}^{(i)} = \phi_{\delta_i, r_i, s_i,t_i}^{(i)}$ by the theory of complex multiplication for abelian varieties  (See \Cref{hecke character}) such that
$$L(s,\jac(C_i)^{\rm{new}})=L(s-\frac{1}{2}, \phi_{\delta}^{(i)}),$$
and it is a Hecke $L$-function. We also have the complete $L$-function for $\jac(C_i)^{\rm{new}}$ as 
$$\Lambda(s, \jac(C_i)^{\rm{new}})=N_{r_i,s_i,t_i,\delta_i}^{\frac{s}{2}}\Gamma_{\C}(s)^{\frac{\ell^{i-1}(\ell-1)}{2}}L(s, \jac(C_i)^{\rm{new}})$$
where 
$\Gamma_{\C}(s)=2(2\pi)^{-s}\Gamma (s)$ and $N_{r_i,s_i,t_i,\delta_i}=\abs{D_{\Q(\zeta_{\ell^i})}}\mc{N}(\mf{f}_{\phi_{\delta}^{(i)}})$ is the global conductor for the $L$-function $L(s, \phi_{\delta}^{(i)})$ (cf. \cite[Theorem 4.60]{kahn2020zeta}, see \Cref{global conductor} for an explicit formula). By the functional equation of the Hecke $L$-function (cf. \cite[Theorem 4.60]{kahn2020zeta}), we have
$$ \Lambda(s, \phi_{\delta}^{(i)}) = W(\phi_{\delta}^{(i)}) \Lambda(1- s, \overline{\phi_{\delta}^{(i)}})$$
where $$\Lambda(s, \phi_{\delta}^{(i)})=N_{r_i,s_i,t_i,\delta_i}^{\frac{s}{2}}\Gamma_{\C}\left(s+\frac{1}{2}\right)^{\frac{\ell^{i-1}(\ell-1)}{2}}L(s, \phi_{\delta}^{(i)})$$ is the complete Hecke $L$-function and $W(\phi_{\delta}^{(i)}) = \pm 1$ is the root number of $\phi_{\delta}^{(i)}$. Thus, we have a functional equation for $ \Lambda(s, \jac(C_i)^{\rm{new}}) $:
$$ \Lambda(s, \jac(C_i)^{\rm{new}}) = W(\phi_{\delta}^{(i)}) \Lambda(2 - s, \jac(C_i)^{\rm{new}}). $$
Since the root number of $J_N$ is the product of 
those of $\phi^{(i)}_\delta$'s, we may focus on the case $i=N$ as an essential case.

According to the Birch and Swinnerton-Dyer conjecture, the parity of the Mordell-Weil rank of $\jac(C_N)^{\rm{new}}(\Q)$ is determined by its root number (cf. \cite[Conjecture 1.1]{Dokchitser2023}). This suggests that the root number encodes important arithmetic data. We review previous work for the root numbers of twisted Fermat quotient curves. The case $N = \delta = 1$ was treated in \cite{gross1978some}, and Shu \cite{shu2021root} extended it to general $\delta$ using Rohrlich's formula \cite[Proposition 2]{rohrlich1992root}. Stoll \cite{stoll2002arithmetic} also computed the root number for $C_1^{(\delta;1,1,\ell-2)}\ (2\ell\nmid \delta)$ by using classical methods such as Eisenstein's reciprocity law. 

In this paper, we extend Shu's result to general $N$ under mild conditions on $r,s,t$ and $\delta$. Our result described below can also be regarded as an extension of Stoll's result \cite{stoll2002arithmetic}.

\begin{theoremeigo}\label{main}
Let $r, s, t > 0$ be integers such that $r + s + t = \ell^N$ and $\ell \nmid rst$. Suppose $$v_{\ell}=\ord_{\ell}\left\{\left(\frac{r^rs^s(\ell^N-t)^t \delta^{r+s}}{\ell^{(r+s)\ord_{\ell}(\delta)}}\right)^{\ell-1}-1\right\}$$
    and when $\ell\nmid \delta$,
    $$c_0=-\frac{\left(r^rs^s(\ell^N-t)^t \delta^{r+s}\right)^{\ell-1}-1}{\ell^{v_{\ell}}}.$$
Then the global root number of $\phi_{\delta}^{(N)}$ is given by
$$W(\phi_{\delta}^{(N)})=\dprod_{\text{$p\le \infty$}}W_p(\phi_{\delta}^{(N)})$$
where for primes $p \neq \ell$,
$$W_p(\phi_{\delta}^{(N)})
=\begin{dcases}
	i^{-\frac{\ell^{N-1}(\ell-1)}{2}}\ &(p=\infty)\\
	\leg{p}{\ell}\ &(p\mid \delta)\\
	1\ &(p\nmid \delta)
\end{dcases} $$
and,
$$W_{\ell}(\phi_{\delta}^{(N)})=
\begin{dcases}
	-\leg{(-1)^{N+1}rst\ord_{\ell}(\delta)(r+s)}{\ell}i^{\frac{\ell^{N-1}(\ell-1)}{2}}\ &(\ell\mid\delta)\\
	-\leg{(-1)^{N+1} rstc_0}{\ell}i^{\frac{\ell^{N-1}(\ell-1)}{2}}\ &(\ell\nmid\delta,\ 1\leq v_{\ell}\leq N)\\
	\leg{2}{\ell}i^{\frac{\ell^{N-1}(\ell-1)}{2}}\ &(\ell\nmid\delta,\ v_{\ell}>N).
\end{dcases}$$
\end{theoremeigo}

We also give the explicit formula for the global conductor $N_{r,s,t,\delta}$ for the $L$-function $L(s, \phi_{\delta}^{(N)})$ as follows:
\begin{theoremeigo}[\Cref{global conductor}]
    Put
    \[
        v
        :=
        \ord_{\ell}
        \left\{
            \left(
                \frac{
                    r^rs^s(\ell^N-t)^t\delta^{r+s}
                }{
                    \ell^{(r+s)\ord_{\ell}(\delta)}
                }
            \right)^{\ell-1}
            -1
        \right\}.
    \]
    Then the global conductor is
    \[
        N_{r,s,t,\delta}
        =
        \ell^{\ell^{N-1}(N\ell-N-1)+f_{\ell}}
        \prod_{\substack{p\mid\delta\\p\neq\ell}}
        p^{\ell^{N-1}(\ell-1)},
    \]
    where
    \[
        f_{\ell}
        =
        \begin{dcases}
            \ell^{N-1}(\ell+1)
            &(\ell\mid \delta)\\
            2\ell^{N-v}
            &(\ell\nmid \delta,\ 1\leq v\leq N)\\
            1
            &(\ell\nmid \delta,\ v>N).
        \end{dcases}
    \]
\end{theoremeigo}
 The calculation of the root number is reduced to computing certain special values of the $\ell^N$-th Hilbert symbol at ramified primes and it requires complicated combinatorial arguments for $N\geq 2$ which never appear when $N=1$. Our research has revealed that it can be expressed in terms of \textit{Fleck numbers}, certain alternating sums of binomial coefficients (See Section 4.3.2). Fleck numbers have been studied in detail in combinatorics (cf. \cite{Sun2006, Sun-Wan2007, wan2006combinatorial, gessel2001order, lundell1978divisibility, neuberger2020elementary}). This phenomenon is of independent interest.

We organize this paper as follows. In Section 2, we compute the $L$-function of $C_N$ and we identify it with the product of the $L$-functions of Hecke characters $\phi_{\delta}^{(i)}\ (1 \leq i \leq N)$. The global root number of $\phi_{\delta}^{(i)}$ is given by the product of the local root numbers $W(\phi_{\delta, V}^{(i)})$ for primes $V$ of $\Q(\zeta_{\ell^i})$ where $\zeta_{\ell^i}:=e^{2\pi i/\ell^i}$:
$$ W(\phi_{\delta}^{(i)}) = \prod_{p\leq \infty} \prod_{V \mid p} W(\phi_{\delta, V}^{(i)}).$$
In the following sections we compute the product $W_p(\phi_{\delta}^{(i)}) := \prod_{V \mid p} W(\phi_{\delta, V}^{(i)})$ for each rational prime $p\leq \infty$.
After the preliminaries in Section 4.1, we compute $W_p(\phi_{\delta}^{(i)})\ (p \neq \ell)$ in Section 4.2. Then, in Section 4.3, we handle the case where $p = \ell$ and this involves complicated combinatorial arguments. The calculation in the latter case is based on Rohrlich's formula \cite[Proposition 2]{rohrlich1992root} as in Shu \cite{shu2021root}. For this purpose, we use Sharifi's formula \cite[Theorem 8]{sharifi2001norm} to determine the conductor of the $\ell^N$-th Hilbert symbol and Coleman's formula \cite[p. 89, line 2]{coleman1988stable} to compute the required special values.

\textbf{Acknowledgments}

The author expresses deep gratitude to his supervisor, Takuya Yamauchi, for his dedicated guidance and encouragement. He carefully guided the author on the subject of this paper. The author also thanks the members of Yamauchi's laboratory and Yasuo Ohno's laboratory for their support during his studies. In particular, the author especially thanks Mahiro Yokomizo for helpful discussions.

The author also expresses his gratitude to his former supervisor, Tamotsu Ikeda (Kyoto University) and the members of his laboratory.

Finally, the author would like to thank his homeroom teacher in middle and high school, Takashi Sakurai, for sparking his interest in mathematics, the peers he attended seminars with at Hiroshima University, and his family for their financial support and understanding of his continued pursuit of mathematics.

The author is supported by JST SPRING, Grant Number JPMJSP2114.

\textbf{Notation.}

\begin{itemize}
	\item Let $\ell$ be an odd prime. Let $N \geq 1$ be an integer, and $\delta \geq 1$ be an $\ell^N$-th power free integer (i.e. for any prime $p$, $\delta$ is not divisible by $p^{\ell^N}$), satisfying either $\ord_{\ell}(\delta) = 0$ or $\ell \nmid \ord_{\ell}(\delta)$.
        \item For a prime $p$ and integer $k,\ n$, $p^k\| n$ means $p^k\mid n$ and $p^{k+1}\nmid n$.
	\item Let $\zeta_m = e^{2\pi i / m}$, and in particular, $\zeta := \zeta_{\ell^N}$ for simplicity, and let $K = \Q(\zeta)$ and $F = \Q(\zeta + \zeta^{-1})$.
        \item Let $\mc{O}_K$ (resp. $\mc{O}_F$) be the ring of integers of $K$ (resp. $F$). For an ideal $\mf{a}$ of the ring of integers $\mc{O}_L$ of some number field $L$, $\mc{N}(\mf{a})$ denotes its absolute norm (i.e. $\mc{N}(\mf{a})=\# (\mc{O}_L/\mf{a})$). 
        \item For a character $\chi$ of the multiplicative group of a finite field $\F_q$, we define its extension to $\F_q$ as usual:
        $$\chi(0):=
        \begin{dcases}
            1 &(\chi=1)\\
            0 &(\chi\neq 1).
        \end{dcases}$$
        \item For characters $\chi$ and $\psi$ of the multiplicative group of a finite field $\F_q$, we write their Jacobi sum as
$$
J(\chi,\psi) = \sum_{a \in \F_q} \chi(a) \psi(1-a).
$$
	\item Let $\pi = \pi_K = \zeta - \zeta^{-1}$ and $\pi' = 1 - \zeta$. We note that $\pi$ is a prime element of $K$ and satisfies $\pi_F := \pi^2 \in \mathcal{O}_F$.
	\item For a local field $M/\Q_p$ containing $n$-th primitive roots of unity, let $(a, b)_n=(a, b)_{n,M}$ be the $n$-th Hilbert symbol of $a,b\in M\x$ defined as
    $$(a, b)_n:=\frac{(a,M(\sqrt[n]{b})/M)(\sqrt[n]{b})}{\sqrt[n]{b}}$$
    where $(\cdot ,M(\sqrt[n]{b})/M):M\x\to\gal(M(\sqrt[n]{b})/M)$
    is the local reciprocity map.
    We write $[a, b]_n \in \Z / n\Z$ for the residue class such that $(a, b)_n = \zeta_n^{[a, b]_n}$.
	\item For $n,r\in\Z_{\geq 0}$ with $\ n\geq r$, let $\displaystyle\binom{n}{r}$ be the binomial coefficient and it is extended to all integers $r$ by $\displaystyle\binom{n}{r} := 0$  if $n < r$ or $r < 0$.
    \item For abelian varieties $A,B$ defined over $\Q$, we write $A\stackrel{\Q}{\sim}B$ if $A$ is isogenous to $B$ over $\Q$.
    \item For a real number $a$, $\lfloor a\rfloor$, $\lceil a\rceil$, and
$\{a\}:=a-\lfloor a\rfloor$ denote the integer part, the ceiling,
and the fractional part of $a$, respectively.
    \item For a proposition $P(n)$ of an integer variable $n$, $\1_{P(n)}$ is the indicator function. That is, it is $1$ if $P(n)$ is true and it is $0$ if $P(n)$ is false. 
\end{itemize}

\begingroup
\raggedbottom

\section{$L$-function for the curve $C_N$}

In this section, we review the definition of $L$-functions for abelian varieties and curves. Then, we compute the $L$-polynomials for $C_N$.

\subsection{\ $L$-functions for abelian varieties}\label{l for abelian}

In this section, we refer to \cite[p. 180, (d)]{milne1972arithmetic}. Let $F$ be a number field and $A$ be an abelian variety over $F$. Let $V_{\ell}=T_{\ell}(A)\otimes_{\Z_{\ell}} \Q_{\ell}$ be the $\ell$-adic rational Tate module of $A$. Let the $\rho_{\ell,A}:\gal(\overline{F}/F)\to \GL_{\Q_{\ell}}(V_{\ell})$ be the $\ell$-adic Galois representation associated to $V_{\ell}$. For each finite prime $\mf{p}\nmid\ell$, let $D_{\mf{p}}$ (resp. $I_{\mf{p}}$) be the decomposition group at $\mf{p}$ (resp. the inertia group at $\mf{p}$) and $\frob_{\mf{p}}\in D_{\mf{p}}$ be the arithmetic Frobenius.
Then we define the Euler factor of the $L$-function of $A$ at $\mf{p}$
$$L_{\mf{p}}(s,A):=L_{\mf{p}}(s,\rho_{\ell,A}):=\det(1_{{V_{\ell}}^{I_{\mf{p}}}}-\rho_{\ell,A}(\frob_{\mf{p}})|_{{V_{\ell}}^{I_{\mf{p}}}} \mc{N}(\mf{p})^{-s})\inv$$
where $\mathcal N$ denotes the absolute norm and $s\in\C$. This Euler factor is independent of the choice of $\ell$.
When $F=\Q$,\ $L_{(p)}(s,A)$ is written as $L_{p}(s,A)$.
Now we define the $L$-function of $A$ as follows:
$$L(s,A)=\prod_{\mf{p}} L_{\mf{p}}(s,A)$$
where $\mf{p}$ runs over all finite primes of $F$.

\subsection{\ $L$-function for nonsingular projective algebraic curves}\label{l for curve}
Let us keep the notation in \Cref{l for abelian}. Let $X$ be a nonsingular projective algebraic curve over $F$, $\jac(X)$ be the Jacobian variety of $X$ and let $L_{\mf{p}}(s,X):=L_{\mf{p}}(s,\rho_{\ell,\jac(X)})$. We also define the $L$-function of $X$ as $$L(s,X)=\prod_{\mf{p}} L_{\mf{p}}(s,X).$$
For each finite prime $\mf{p}$ such that $X$ has good reduction at $\mf{p}$, it is known that the zeta function $Z(T,X_{/\F_{\mf{p}}})$ of $X_{/\F_{\mf{p}}}$ (where $\F_{\mf{p}}$ is the residue field of $\mf{p}$ and its cardinality is $q$) has the form
$$Z(T,X_{/\F_{\mf{p}}})=\frac{L(T,X_{/\F_{\mf{p}}})}{(1-T)(1-qT)}=Z(T,{\P^1}_{/\F_q})L(T,X_{/\F_{\mf{p}}})$$ with some integral polynomial $L(T,X_{/\F_{\mf{p}}})$ (called the $L$-polynomial) and $L_{\mf{p}}(s,X)=L(q^{-s},X_{/\F_{\mf{p}}})\inv\ (s\in\C)$.

We note that
\[
		L(T,X_{/\F_{\mf{p}}})
        =\frac{Z(T,X_{/\F_{\mf{p}}})}{Z(T,\P^1_{/\F_{\mf{p}}})}
        =\exp\left(\sum_{k\geq 1} (\# X(\F_{q^k})-q^k-1)\frac{T^k}{k}\right).
	\]

\subsection{\ Computation of the $L$-polynomial for the curve $C_N$}

Let $r, s, t > 0$ be integers such that $r + s + t = \ell^N$ and $\ell \nmid rst$.
Let $C_N$ be the nonsingular model of the projective closure $\widetilde{C}^0_N$ of the affine curve $C^0_N:y^{\ell^N} = x^r(\delta - x)^s$ which is also defined over $\Q$.
To compute the $L$-polynomial, we first need the formula for $\# C_N(\F_q)$ for finite fields $\F_q$ as follows: 
\begin{propositioneigo}\label{count}
	Let $p$ be a prime such that $C_N$ has good reduction, namely $p\nmid \delta \ell$ and $q$ be a power of $p$. Then

	$\# C_N(\F_q)=
	\begin{dcases}
		q+1\ &(\ell\nmid q-1)\\
		q+1+\dsum_{\substack{\chi^{\ell^N}=1\\ \chi\neq 1}} \chi(\delta)^{r+s}J(\chi^r,\chi^s)\ &(\ell\mid q-1)
	\end{dcases}$

	where $\chi$ runs over the characters of $\F_q\x$.
\end{propositioneigo}
\begin{proof}
    $\widetilde{C}^0_N$ can be seen as the gluing curve of the two affine curves $C^0_N$ and ${C^0_N}':w^{\ell^N}=v^{\ell^N-r-s}(v\delta-1)^s$ by a rational map $(x,y)\mapsto (\frac{1}{v},\frac{w}{v})$.
    The only possible singular points of $\widetilde{C}^0_N$ are \((0,0)\), \((\delta,0)\), and the infinite point  $(v,w)=(0,0)$ on ${C^0_N}'$. Since
\[
\gcd(\ell^N,r)=\gcd(\ell^N,s)=\gcd(\ell^N,t)=1,
\]
the normalization has exactly one point above each of \(x=0\),
\(x=\delta\), and \(x=\infty\). These points are
\(\mathbb F_q\)-rational. Thus, normalization does not change the
number of \(\mathbb F_q\)-rational points, and
\[
\#C_N(\mathbb F_q)=\# \widetilde{C}^0_N(\F_q)=\#C_N^0(\mathbb F_q)+1.
\]
	 Thus, the claim follows together with the following computation:
	\begin{align}
		\# C^0_N(\F_q)=& \sum_{u\in\F_q} \# \Set{y\in\F_q}{y^{\ell^N}=u^r(\delta-u)^s}\\
		=& \sum_{u\in\F_q} \sum_{\chi^{\ell^N}=1} \chi(u^r(\delta-u)^s)\\
		=& \sum_{\chi^{\ell^N}=1} \sum_{u\in\F_q} \chi(\delta)^{r+s}\chi({u'}^r(1-u')^s)\ (u:=\delta u')\\
		=& \sum_{\chi^{\ell^N}=1} \chi(\delta)^{r+s}\sum_{u\in\F_q} \chi^r(u')\chi^s(1-u')\\
		=& q+\dsum_{\substack{\chi^{\ell^N}=1\\ \chi\neq 1}} \chi(\delta)^{r+s}J(\chi^r,\chi^s).
	\end{align}
\end{proof}


Let $f=f_p$ be the order of a prime $p\neq \ell$ in $(\Z/\ell\Z)^\times$ and let $e_p = \ord_{\ell}(p^f - 1)$. We note that $e_p$ could be strictly greater than 1 (for instance, when $p = 3$ and $\ell = 11$, we have $f = 5$ and $e_p = 2$). 


The following lemma is easily proved and we omit a proof.

\begin{lemmaeigo}\label{order}
	For each integer $t \geq 1$ and each prime $p\neq \ell$, the order of $p$ in $(\Z / \ell^t \Z)^\times$ is given by $\ord(p \pmod{\ell^t}) = \ell^{N_t} f$, where
	$$ N_t =
	\begin{dcases}
		0 & (1 \leq t \leq e_p) \\
		t - e_p & (t > e_p).
	\end{dcases} $$
\end{lemmaeigo}

Here, we review the power residue symbol. For an integer $n\geq 1$ and a prime ideal $\mf{p}\nmid n$ of a number field $K\supset \Q(\zeta_{n})$ and $\alpha\in\mc{O}_K\setminus{\mf{p}}$, the $n$-th power residue symbol $\leg{\alpha}{\mf{p}}_n$ is the unique $n$-th root of unity which satisfies $$\alpha^{\frac{\mc{N}\mf{p}-1}{n}}\equiv \leg{\alpha}{\mf{p}}_n \pmod{\mf{p}}.$$

Now, we prepare the family of characters $\{\chi_{p^{kf}, \ell^t}\}_{k \geq 1, 1 \leq t \leq e_p + \ord_{\ell}(k)}$ where $p$ is a prime number, $\chi_{p^{kf}, \ell^t}$ is a character of $\F_{p^{kf}}^\times$ and its order is $\ell^t$ as follows:

First, we prepare a family of prime ideals $\{\mf{p}_t\}_{t\geq 1}$ such that $\mf{p}_t$ is a prime ideal of $\Q(\zeta_{\ell^t})$ over $p$, and $\mf{p}_{t+1}\cap\Q(\zeta_{\ell^t})=\mf{p}_t$ for all $t$.

Next, we define $\chi_{p^{\ell^j f}, \ell^t}$ for $j \geq 0$ and $1 \leq t \leq e_p + j$ as follows:

$$\chi_{p^{\ell^j f}, \ell^t} := 
\begin{dcases}
    \leg{\cdot}{\mf{p}_t}_{\ell^t}\ &(j=0)\\
    \leg{\cdot}{\mf{p}_{j+e_p}}_{\ell^t}\ &(j>0).
\end{dcases}$$

Finally, we define $\chi_{p^{kf}, \ell^t}$ for integer $k \geq 1$ and $1 \leq t \leq e_p + j$ where $j:=\ord_{\ell}(k)$ as $$\chi_{p^{kf}, \ell^t} := \chi_{p^{\ell^j f}, \ell^t} \circ N_{\F_{p^{kf}} / \F_{p^{\ell^j f}}}.$$
Here, these characters satisfy the following compatibility:
\begin{propositioneigo}\label{norm compati}
    Let $K = \Q(\zeta_{\ell^{n+1}}),\ F= \Q(\zeta_{\ell^n})$. Let $\mf{P}$ be an unramified prime ideal of $K$, and $\mf{p}=\mf{P}\cap F$ be the prime ideal of $F$ lying below $\mf{P}$.
	For $\alpha \in \mc{O}_{K}\setminus{\mf{P}}$, we have
	$$ \leg{\alpha}{\mf{P}}_{\ell^n} = \leg{N_{\F_{\mf{P}} / \F_{\mf{p}}}(\alpha)}{\mf{p}}_{\ell^n}. $$
	Therefore, in particular,
	$$ \chi_{p^{\ell^j f}, \ell^{e_p + j}}^{\ell} = \chi_{p^{\ell^{j-1} f}, \ell^{e_p + j - 1}} \circ N_{\F_{p^{\ell^j f}} / \F_{p^{\ell^{j-1} f}}} \quad (j \geq 1). $$
\end{propositioneigo}
\begin{proof}
	A prime element $\pi$ of $F_{\mf{p}}$ is also a prime element of $K_{\mf{P}}$, so, by the Hilbert symbol lifting formula \cite[Corollary]{bender1973lifting}, we have
	\begin{align}
		\leg{\alpha}{\mf{P}}_{\ell^n}=(\pi, \alpha)_{K_{\mf{P}}, \ell^n} &= (\pi, N_{K_{\mf{P}} / F_{\mf{p}}}(\alpha))_{F_{\mf{p}}, \ell^n} = \leg{N_{K_{\mf{P}} / F_{\mf{p}}}(\alpha)}{\mf{p}}_{\ell^n}.
	\end{align}


    Here, since $K_{\mf{P}} / F_{\mf{p}}$ is unramified, we have $\gal(K_{\mf{P}} / F_{\mf{p}}) \cong \gal(\F_{\mf{P}}/ \F_{\mf{p}})$, and
	$$ [N_{K_{\mf{P}} / F_{\mf{p}}}(\alpha)] = N_{\F_{\mf{P}} / \F_{\mf{p}}}([\alpha])$$
    in $\F_{\mf{p}}$ for $\alpha\in\mc{O}_{K}\setminus{\mf{P}}$.
\end{proof}
From the above, we can compute the $L$-function:
\begin{propositioneigo}\label{l function}
    Let $p$ be a prime such that $C_N$ has good reduction. The $L$-polynomial of $C_N$ at $p$ is given by
    $$L_{p}(T,{C_N}_{/\Q})
    =\prod_{1\leq j\leq N} \prod_{\substack{1\leq i\leq \ell^j-1\\ \ell\nmid i}} \left(1+\chi_{p^{\ell^{N_j}f},\ell^j}(\delta)^{(r+s)i}J\left(\chi_{p^{\ell^{N_j}f},\ell^j}^{ri},\chi_{p^{\ell^{N_j}f},\ell^j}^{si}\right)T^{\ell^{N_j}f}\right)^{\frac{1}{\ell^{N_j}f}}$$
    where the integer $N_j$ is defined in \Cref{order}. The $L$-function of $C_N$ at $p$ is given by
    $$L_{p}(s,{C_N}_{/\Q})
    =\prod_{1\leq j\leq N} \prod_{\substack{\mf{p}_j\mid p}} \left(1+\leg{\delta}{\mf{p}_j}_{\ell^j}^{r+s} J\left(\leg{\cdot}{\mf{p}_j}_{\ell^j}^r,\leg{\cdot}{\mf{p}_j}_{\ell^j}^s\right)\mc{N}\mf{p}_j^{-s}\right)\inv .$$
    Here, $\mf{p}_j$ runs over all prime ideals of $\Q(\zeta_{\ell^j})$ lying above $p$.
\end{propositioneigo}
\begin{proof}

    By \Cref{l for curve}, we have
    \begin{align}
        \log (L_{p}(T,{C_N}_{/\Q}))
        &=\sum_{n\geq 1} (\# C_N(\F_{p^{n}})-p^{n}-1)\frac{T^{n}}{n}.
    \end{align}
    By Proposition 2.1, $\# C_N(\F_{p^{n}})-p^{n}-1= 0$ unless $p^n\equiv 1\pmod{\ell}$. Hence only the integers $f\mid n$ contribute. Therefore, setting $n=kf$, we get
    $$\log (L_{p}(T,{C_N}_{/\Q}))
    =\sum_{k\geq 1} (\# C_N(\F_{p^{kf}})-p^{kf}-1)\frac{T^{kf}}{kf}
    =\sum_{k\geq 1} \dsum_{\substack{\chi^{\ell^N}=1\\ \chi\neq 1}} \chi(\delta)^{r+s}J(\chi^r,\chi^s)\frac{T^{kf}}{kf}.$$
    Since $\chi$ runs over the characters of $\F_{p^{kf}}\x$, from \Cref{count}, if we set $t_k=\min\{e_p+\ord_{\ell}(k), N\}$, we have
    $$\log (L_{p}(T,{C_N}_{/\Q}))
    =\sum_{k\geq 1} \sum_{1\leq i\leq \ell^{t_k}-1} \chi_{p^{kf}, \ell^{t_k}}(\delta)^{(r+s)i} J(\chi_{p^{kf}, \ell^{t_k}}^{ri}, \chi_{p^{kf}, \ell^{t_k}}^{si})\frac{T^{kf}}{kf}.$$
    Therefore, we must consider the form $t_k$ takes for each $k$.

    \textbf{Case (i) $N\geq e_p$.}

	In this case,
    $$t_k=
    \begin{dcases}
        e_p+j\ &(\text{for $0\leq j\leq N-e_p-1$, $\ell^j\| k$})\\
        N\ &(\ell^{N-e_p}\mid k)
    \end{dcases}$$
    therefore, we have
    \begin{align}
		&\log (L_{p}(T,{C_N}_{/\Q}))\\
		=&\sum_{0\leq j\leq N-e_p-1}\sum_{\ell^j\| k} \left(\sum_{1\leq i\leq \ell^{e_p+j}-1} \chi_{p^{kf}, \ell^{e_p+j}}(\delta)^{(r+s)i} J(\chi_{p^{kf}, \ell^{e_p+j}}^{ri}, \chi_{p^{kf}, \ell^{e_p+j}}^{si})\right)\frac{T^{kf}}{kf}\\
		&+\sum_{\ell^{N-e_p}\mid k} \left(\sum_{1\leq i\leq \ell^{N}-1} \chi_{p^{kf}, \ell^{N}}(\delta)^{(r+s)i} J(\chi_{p^{kf}, \ell^{N}}^{ri}, \chi_{p^{kf}, \ell^{N}}^{si})\right)\frac{T^{kf}}{kf}\\

		=&\sum_{0\leq j\leq N-e_p-1}\sum_{\ell\nmid k'} \left(\sum_{1\leq i\leq \ell^{e_p+j}-1} \chi_{p^{\ell^j k'f}, \ell^{e_p+j}}(\delta)^{(r+s)i} J(\chi_{p^{\ell^j k'f}, \ell^{e_p+j}}^{ri}, \chi_{p^{\ell^j k'f}, \ell^{e_p+j}}^{si})\right)\frac{T^{\ell^j k'f}}{\ell^j k'f}\\
		&+\sum_{k'\geq 1} \left(\sum_{1\leq i\leq \ell^{N}-1} \chi_{p^{\ell^{N-e_p}k'f}, \ell^{N}}(\delta)^{(r+s)i} J(\chi_{p^{\ell^{N-e_p}k'f}, \ell^{N}}^{ri}, \chi_{p^{\ell^{N-e_p}k'f}, \ell^{N}}^{si})\right)\frac{T^{\ell^{N-e_p}k'f}}{\ell^{N-e_p}k'f}.
	\end{align}
    Note that $k=\ell^jk'$ in the first term,\ $k=\ell^{N-e_p}k'$ in  the second one in the last line.
	Here, by the Hasse--Davenport relation, we have $J(\chi_{p^{\ell^j k'f}, \ell^{e_p+j}}^{ri}, \chi_{p^{\ell^j k'f}, \ell^{e_p+j}}^{si})=(-1)^{k'-1}J(\chi_{p^{\ell^j f}, \ell^{e_p+j}}^{ri}, \chi_{p^{\ell^j f}, \ell^{e_p+j}}^{si})^{k'}$, and by the definition of the family \\
	$\{\chi_{p^{kf}, \ell^t}\}_{k\geq 1, 1\leq t\leq e_p+\ord_{\ell}(k)}$, we get $\chi_{p^{\ell^j k'f}, \ell^{e_p+j}}(\delta)=\chi_{p^{\ell^j f}, \ell^{e_p+j}}(N_{\F_{p^{\ell^{j}k' f}}/\F_{p^{\ell^{j} f}}}(\delta))=\chi_{p^{\ell^j f}, \ell^{e_p+j}}(\delta)^{k'}$.
	Therefore,
	\begin{align}
		&\log (L_{p}(T,{C_N}_{/\Q}))\\
		=&\sum_{0\leq j\leq N-e_p-1}\frac{1}{\ell^j f}\sum_{\ell\nmid k'} \sum_{1\leq i\leq \ell^{e_p+j}-1} \frac{(-1)^{k'-1}}{k'}\left(\chi_{p^{\ell^j f}, \ell^{e_p+j}}(\delta)^{(r+s)i} J(\chi_{p^{\ell^j f}, \ell^{e_p+j}}^{ri}, \chi_{p^{\ell^j f}, \ell^{e_p+j}}^{si})T^{\ell^jf}\right)^{k'}\\
		&+\frac{1}{\ell^{N-e_p} f}\sum_{k'\geq 1} \sum_{1\leq i\leq \ell^{N}-1} \frac{(-1)^{k'-1}}{k'}\left(\chi_{p^{\ell^{N-e_p}f}, \ell^{N}}(\delta)^{(r+s)i} J(\chi_{p^{\ell^{N-e_p}f}, \ell^{N}}^{ri}, \chi_{p^{\ell^{N-e_p}f}, \ell^{N}}^{si})T^{\ell^{N-e_p}f}\right)^{k'}.
	\end{align}

	Since $\ell$ is odd, from $\dsum_{\substack{k\geq 1\\ \ell\nmid k}} (-1)^{k-1}\dfrac{T^k}{k}=\log(1+T)-\dfrac{1}{\ell}\log(1+T^{\ell})$, we get
	\begin{align}
		&\log (L_{p}(T,{C_N}_{/\Q}))\\
		=&\sum_{0\leq j\leq N-e_p-1}\frac{1}{\ell^j f}\sum_{1\leq i\leq \ell^{e_p+j}-1} \biggl\{\log\left(1+\chi_{p^{\ell^j f}, \ell^{e_p+j}}(\delta)^{(r+s)i} J(\chi_{p^{\ell^j f}, \ell^{e_p+j}}^{ri}, \chi_{p^{\ell^j f}, \ell^{e_p+j}}^{si})T^{\ell^jf}\right)\\
		&-\frac{1}{\ell}\log\left(1+\left(\chi_{p^{\ell^j f}, \ell^{e_p+j}}(\delta)^{(r+s)i} J(\chi_{p^{\ell^j f}, \ell^{e_p+j}}^{ri}, \chi_{p^{\ell^j f}, \ell^{e_p+j}}^{si})T^{\ell^jf}\right)^{\ell}\right)\biggr\}\\
		&+\frac{1}{\ell^{N-e_p} f}\sum_{1\leq i\leq \ell^{N}-1} \log\left(1+\chi_{p^{\ell^{N-e_p}f}, \ell^{N}}(\delta)^{(r+s)i} J(\chi_{p^{\ell^{N-e_p}f}, \ell^{N}}^{ri}, \chi_{p^{\ell^{N-e_p}f}, \ell^{N}}^{si})T^{\ell^{N-e_p}f}\right).
	\end{align}

	From \Cref{norm compati}, we have $$\chi_{p^{\ell^{j+1} f}, \ell^{e_p+j+1}}(\delta)^{\ell}=\chi_{p^{\ell^j f}, \ell^{e_p+j}}(N_{\F_{p^{\ell^{j+1} f}}/\F_{p^{\ell^{j} f}}}(\delta))=\chi_{p^{\ell^j f}, \ell^{e_p+j}}(\delta)^{\ell}.$$ Also, since $\ell$ is odd, by the Hasse--Davenport relation and \Cref{norm compati}, we have
	\begin{align}
		&\left(\chi_{p^{\ell^j f}, \ell^{e_p+j}}(\delta)^{(r+s)i} J(\chi_{p^{\ell^j f}, \ell^{e_p+j}}^{ri}, \chi_{p^{\ell^j f}, \ell^{e_p+j}}^{si})T^{\ell^jf}\right)^{\ell}\\
		=&\chi_{p^{\ell^{j+1} f}, \ell^{e_p+j+1}}(\delta)^{(r+s)i\ell} J(\chi_{p^{\ell^{j+1} f}, \ell^{e_p+j+1}}^{r\ell i}, \chi_{p^{\ell^{j+1} f}, \ell^{e_p+j+1}}^{s\ell i})T^{\ell^{j+1}f}.
	\end{align}



	Thus, we obtain the following expression:
	\begin{align}
		&\log (L_{p}(T,{C_N}_{/\Q}))\\
		=&\frac{1}{f}\sum_{1\leq i\leq \ell^{e_p}-1} \log\left(1+\chi_{p^{f}, \ell^{e_p}}(\delta)^{(r+s)i} J(\chi_{p^{f}, \ell^{e_p}}^{ri}, \chi_{p^{f}, \ell^{e_p}}^{si})T^{f}\right)\\
		&+\sum_{1\leq j\leq N-e_p}\frac{1}{\ell^j f}\sum_{\substack{1\leq i\leq \ell^{e_p+j}-1\\ \ell\nmid i}} \log\left(1+\chi_{p^{\ell^j f}, \ell^{e_p+j}}(\delta)^{(r+s)i} J(\chi_{p^{\ell^j f}, \ell^{e_p+j}}^{ri}, \chi_{p^{\ell^j f}, \ell^{e_p+j}}^{si})T^{\ell^jf}\right)\\

		=&\frac{1}{f}\sum_{1\leq j\leq e_p}\sum_{\substack{1\leq i\leq \ell^{j}-1\\ \ell\nmid i}} \log\left(1+\chi_{p^{f}, \ell^{j}}(\delta)^{(r+s)i} J(\chi_{p^{f}, \ell^{j}}^{ri}, \chi_{p^{f}, \ell^{j}}^{si})T^{f}\right)\\
		&+\sum_{1\leq j\leq N-e_p}\frac{1}{\ell^j f}\sum_{\substack{1\leq i\leq \ell^{e_p+j}-1\\ \ell\nmid i}} \log\left(1+\chi_{p^{\ell^j f}, \ell^{e_p+j}}(\delta)^{(r+s)i} J(\chi_{p^{\ell^j f}, \ell^{e_p+j}}^{ri}, \chi_{p^{\ell^j f}, \ell^{e_p+j}}^{si})T^{\ell^jf}\right)\\

		=&\sum_{1\leq j\leq N} \frac{1}{\ell^{N_j}f}\sum_{\substack{1\leq i\leq \ell^j-1\\ \ell\nmid i}} \log \left(1+\chi_{p^{\ell^{N_j}f},\ell^j}(\delta)^{(r+s)i} J\left(\chi_{p^{\ell^{N_j}f},\ell^j}^{ri},\chi_{p^{\ell^{N_j}f},\ell^j}^{si}\right)T^{\ell^{N_j}f}\right)
	\end{align}

	From this, we have obtained the formula for the $L$-polynomial as claimed. Moreover, $\chi_{p^{\ell^{N_j} f}, \ell^j} = \leg{\cdot}{\mf{p}_j}_{\ell^j}$, and the order of the decomposition group of $\mf{p}_j$ is the relative degree of $\mf{p}_j$ over $p$, which is $\ell^{N_j} f$. Thus, we obtain the formula for the $L$-function as claimed.

	\textbf{Case (ii) $N < e_p$.}

	In this case, we have $t_k = N$ for any $k$. That is, $\F_{p^{kf}}$ has a character $\chi_{p^{kf}, \ell^N}$ of order $\ell^N$. Therefore, we have
	\begin{align}
		\log (L_{p}(T,{C_N}_{/\Q}))
	=&\sum_{k\geq 1} \sum_{1\leq i\leq \ell^{N}-1} \chi_{p^{kf}, \ell^{N}}(\delta)^{(r+s)i} J(\chi_{p^{kf}, \ell^{N}}^{ri}, \chi_{p^{kf}, \ell^{N}}^{si})\frac{T^{kf}}{kf}\\
	=&\frac{1}{f} \sum_{1\leq i\leq \ell^{N}-1}\sum_{k\geq 1} \frac{(-1)^{k-1}}{k} \left(\chi_{p^{f}, \ell^{N}}(\delta)^{(r+s)i} J(\chi_{p^{f}, \ell^{N}}^{ri}, \chi_{p^{f}, \ell^{N}}^{si})T^{f}\right)^k\\
	=&\frac{1}{f} \sum_{1\leq i\leq \ell^{N}-1} \log\left(1+\chi_{p^{f}, \ell^{N}}(\delta)^{(r+s)i} J(\chi_{p^{f}, \ell^{N}}^{ri}, \chi_{p^{f}, \ell^{N}}^{si})T^{f}\right)\\
	=&\sum_{1\leq j\leq N} \frac{1}{f}\sum_{\substack{1\leq i\leq \ell^j-1\\ \ell\nmid i}} \log\left(1+\chi_{p^{f}, \ell^{j}}(\delta)^{(r+s)i} J(\chi_{p^{f}, \ell^{j}}^{ri}, \chi_{p^{f}, \ell^{j}}^{si})T^{f}\right).
	\end{align}
    The second line follows from Hasse--Davenport's relation.
	Thus, since $j \leq N < e_p$, we have $N_j = 0$ and thus we have obtained the formula for the $L$-polynomial as claimed. Also, $\chi_{p^f, \ell^j} = \leg{\cdot}{\mf{p}_j}_{\ell^j}$, and the order of the decomposition group of $\mf{p}_j$ is $f$, regardless of $j$. Therefore, we obtain the formula for the $L$-function as claimed.
	\end{proof}
    

\endgroup

  \section{A Hecke character associated to the new part of $\jac (C_N)$}\label{hecke character}
    
Let $1\leq i\leq N$ be an integer. In this section, we define a unitary Hecke character associated to the new part $\jac (C_i)^{\rm new}$ of $J_i=\jac (C_i)$ (See Section \ref{Intro} for the new part) for each $i$. Let $K_i=\Q(\zeta_{\ell^i})$.

We note that $\jac (C_i)^{\rm new}$ has CM by $K_i$ and $K_i$ is embedded into the ring of endomorphisms defined over $K_i$. We also remark that 
$2 \dim \jac (C_i)^{\rm new}=[K_i:\Q]$. 

Applying the results of the theory of complex multiplication for abelian varieties (cf. \cite[Chapter one, Section 1.8, p. 27]{schappacher2006periods}, \cite[Theorem]{weil1952jacobi} or \cite[Introduction]{coleman1988stable}) together with the results in the previous section to 
$\jac(C_i^{(1,r_i,s_i)})^{\rm new}$ (the case when $\delta_i=1$), we can first obtain a unitary Hecke character 
$$\phi^{(i)}_{r_i,s_i,t_i}:\A^\times_{K_i}\longrightarrow \C^\times$$ as follows:
\begin{enumerate}
\item[(i)] for each finite prime $V$ (with the corresponding prime ideal $\frak p_V$) of ${K_i}$ not dividing $\ell {\delta_i}$ 
and each uniformizer $\pi_{i,V}$ of $K_{i,V}$, 
$$\phi^{(i)}_{r_i,s_i,t_i}(\pi_{i,V})=j^{(i)}_{r_i,s_i,t_i}(\mf{p}_V)\mathcal{N}(\frak p_V)^{-\frac{1}{2}}$$ 
where $$
j^{(i)}_{r_i,s_i,t_i}(\mf{p}_V) := -J\left(\leg{\cdot}{\mf{p}_V}_{\ell^i}^{r_i}, \leg{\cdot}{\mf{p}_V}_{\ell^i}^{s_i}\right);
$$
\item[(ii)] the infinite component $\phi^{(i)}_{r_i,s_i,t_i,\infty}:({K_i}\otimes_\Q\R)^\times\longrightarrow 
\C^\times$ of $\phi^{(i)}_{r_i,s_i,t_i}$ is defined by  
$$\phi^{(i)}_{r_i,s_i,t_i,\infty}(x\otimes 1)=\prod_{\sigma\in \Phi^{-1}_i}\left(\frac{\sigma(x)}{|\sigma(x)|}\right)\inv$$ 
where the set $\Phi^{-1}_i$ is the CM type of ${K_i}$ defined by 
$$\Phi_i\inv := \left\{ h\inv \in (\Z/\ell^i\Z)\x \Bigg| \left\{\frac{r_i h}{\ell^i}\right\} + \left\{\frac{s_i h}{\ell^i}\right\} + \left\{\frac{t_i h}{\ell^i}\right\} = 1 \right\}
$$
under the canonical identification ${\rm Gal}({K_i}/\Q)\simeq (\Z/\ell^i\Z)^\times$. 
\end{enumerate}

Let $\chi^{(i)}_{\delta_i}:\A^\times_{K_i}\longrightarrow \C^\times,\ 
x\mapsto \dprod_{V\nmid \infty} (x_V,\delta_i)_{\ell^i, K_{i,V}}$ so that $\chi^{(i)}_{\delta_i,V}(\iota_V(\pi_{i,V}))=\left(\frac{\delta_i}{\frak p_V}\right)_{\ell^i}$ for any finite place $V$ of 
${K_i}$ not dividing $\ell {\delta_i}$ where $\iota_V :K_{i,V}\to \A^\times_{K_i}$ is the canonical embedding (cf. \cite[Chapter V, Proposition 3.1]{neukirch2013algebraic}).  We twist $\phi^{(i)}_{r_i,s_i,t_i}$ by ${\chi^{(i)}_{\delta_i}}^{r_i+s_i}$:
$$\phi^{(i)}_{\delta_i,r_i,s_i,t_i}:=
{\chi^{(i)}_{\delta_i}}^{r_i+s_i} \phi^{(i)}_{r_i,s_i,t_i}.
$$
This is a Hecke character associated with the new part $\jac (C_i)^{\rm new}$ since the Hecke character induced by the Gr\"{o}ssencharacter $\left(\dfrac{\delta_i}{\cdot}\right)_{\ell^i}:I_{K_i}(\delta_i)\to \C\x$ is $\chi^{(i)}_{\delta_i}$ where $I_{K_i}(\delta_i)$ is the group of fractional ideals of $K_i$ that are coprime to $\delta_i$. This follows from the weak approximation theorem (cf. \cite[Proposition 4.7 ]{milneCFT}).

From now on, until the final section, we may focus on the case when $i = N$ without loss of generality, and further, we abbreviate $\phi = \phi^{(N)}_{r,s,t}$, $\phi_{\delta} = \phi^{(N)}_{\delta,r,s,t}$, and $\chi_{\delta} = \chi^{(N)}_{\delta}$.
Then, for any $\delta$ in our setting, we have 
$$L\Big(s+\frac{1}{2},J_N^{\rm new}\Big)=L(s,\phi_{\delta})=L(s,{\chi_{\delta}}^{r+s} \phi)$$
where the right hand side stands for the automorphic $L$-function of 
$\phi_{\delta}$ and the left hand side is defined by the $\ell$-adic 
rational Tate module of $J_N^{\rm new}$. Let ${\rm AI}_{K/\Q}(\phi_{\delta})$ denote the automorphic induction of $\phi_{\delta}$ to $\GL_{d}(\A_\Q)$ 
with $d=[K:\Q]$ 
(cf. \cite{arthur1989simple}). Then, $L(s+\frac{1}{2},J_N^{\rm new})=
L(s,\phi_{\delta})=
L(s,{\rm AI}_{K/\Q}(\phi_{\delta}))$ and both sides are of degree 
$d$. 
As explained in Section \ref{Intro}, we have an isogeny decomposition over $\Q$
$$
J_N \stackrel{\Q}{\sim} \bigoplus_{1 \leq i \leq N} \jac(C_i^{(\delta; r_i, s_i, t_i)})^{\rm{new}}
$$
and accordingly, it yields 
$$L\Big(s+\frac{1}{2},J_N\Big)=
\prod_{1\leq i\leq N} L\Big(s+\frac{1}{2},\jac(C_i^{(\delta; r_i, s_i, t_i)})^{\rm{new}}\Big)=
\prod_{1\leq i\leq N} L(s,\phi^{(i)}_{\delta_i,r_i,s_i,t_i}).$$

Therefore, the root number of $J_N$ can be 
written as the product of those of $\phi^{(i)}_{\delta_i,r_i,s_i,t_i}$'s. 
Thus, the computation is reduced to one for $\phi^{(i)}_{\delta_i,r_i,s_i,t_i}$. 

\section{Root numbers for the Hecke characters}
In this section, we investigate the root number of the Hecke character $\phi_{\delta}$ which corresponds to the new part of $\jac (C_N)$.
\subsection{\ Preparation for Root Numbers}
Let $E$ be a local field. When $E$ is non Archimedean, consider a prime element $\pi_E$ and the normalized absolute value $\abs{\cdot}_E$. Let $\chi$ (resp. $\psi$) be a multiplicative (resp. additive) character of $E$ and $dx$ be the self dual Haar measure on $E$ with respect to $\psi$. Let $U^{(n)}$ be $1+(\pi_E)^n$ for $n>0$, $\mc{O}_E\x$ for $n=0$.

We denote their conductor exponents by $f(\chi)$ and $N(\psi)$, respectively (i.e. each is the minimum integer $f$ (resp. $N$) such that $\chi|_{U^{(f)}}=1$ (resp. $\psi|_{\pi_E^N}=1$)).

In this paper, we choose $\psi$ as follows:

\begin{itemize}
	\item If $E$ is Archimedean, $\psi(x) = e^{-2\pi i \Tr_{E/\R}(x)}$.\\
	\item If $E$ is a $p$-adic field, $\psi(x) = e^{2\pi i \left\{\Tr_{E/\Q_p}(x)\right\}_p}$ where $\{a\}_p$ represents the principal part of the $p$-adic expansion of $a$.
\end{itemize}

Let $\epsilon(\chi, \psi,dx)$ be the epsilon factor of $\chi$ with respect to $\psi$ (See \cite[Theorem 7-2]{ramakrishnan2013fourier}). The local root number of $\chi$ with respect to $\psi$ is denoted by $W(\chi, \psi) := \epsilon(\abs{\cdot}_E^{1/2} \chi, \psi,dx)$. When $\psi$ is clear from the context, it is written simply as $W(\chi)$. For details on root numbers and epsilon factors, we refer to \cite[Chapter 7]{ramakrishnan2013fourier} and \cite[3.1]{shu2021root}.

In the following, we compute the global root number
$$
W(\phi_{\delta}) = \prod_{p\leq \infty} \prod_{V \mid p} W(\phi_{\delta, V}),
$$
where the computation is carried out by calculating the products of local root numbers $W_p(\phi_{\delta}) := \prod_{V \mid p} W(\phi_{\delta, V})$ for each rational prime $p\leq \infty$.

In the next section, we compute $W_p(\phi_{\delta})$ for $p \neq \ell$. For $p = \ell$, we have $W_{\ell}(\phi_{\delta}) = W(\phi_{\delta, \pi})$, and in this case, we use Rohrlich's formula \cite[Proposition 2]{rohrlich1992root} for the computation.

\subsection{\ Calculation of $W_p(\phi_{\delta})$ when $p\neq \ell$}

Let $p$ be a prime of $\Q$ distinct from $\ell$.

\textbf{(a) Calculation of $W_{\infty}(\phi_{\delta})$.}

First, we recall the formulas for $\epsilon$ factor at infinity primes, but be careful that the formulas stated in \cite[p.251-253]{ramakrishnan2013fourier} are wrong, so we propose the correct formula:

\begin{propositioneigo}
    Let $F$ be an Archimedean local field, $\chi$ be a quasi character of $F\x$.
    Write $\chi(z)=\left(\frac{z}{\abs{z}}\right)^m\abs{z}^w\ (m\in\Z, w\in\C)$ (We note that when $F=\R$, we can take $m$ as $m\in\set{0,1}$). Then,

    $$\epsilon(\chi,\psi,dx)=(-i)^{\abs{m}}.$$
\end{propositioneigo}

Since the infinite part of $\chi_{\delta}$ is trivial, we have
\begin{align}
    W_{\infty}(\phi_{\delta})
    = \prod_{V \mid \infty} W_{\infty}(\phi_V)
    = \prod_{V \mid \infty} (-i)^{\abs{-1}}
    = i^{-\frac{\ell^{N-1}(\ell-1)}{2}}.
\end{align}

\textbf{(b) Calculation of $W_p(\phi_{\delta})$ for $p \neq \infty$.}

Let $V$ be a prime of $K$ lying above $p$.
If $V \nmid \delta$, then $\phi_{\delta, V}$ is unramified, therefore, $W(\phi_{\delta, V}) = 1$. Then, it suffices to consider the case when $V$ is ramified (i.e. $V \mid \delta$).

Since $V \nmid \ell$, $\phi_V$ is unramified. Thus, by \cite[Chapter 7 Exercises 8 (c), p. 301]{ramakrishnan2013fourier}, we have
$$
W(\phi_{\delta, V}) = \phi_V(\pi_V)^{-N(\psi) + f(\phi_{\delta, V})} W(\chi_{\delta, V}^{r+s}) = \phi_V(\pi_V) W(\chi_{\delta, V}^{r+s})
$$
where $\pi_V$ denotes a uniformizer of $K_V$. Now, we have to examine those two terms respectively.

\begin{propositioneigo}
Let $v$ be the prime of $F$ lying below $V$. 
\begin{enumerate}
    \item[(a)] If $v$ is inert in $K/F$, then
    \[
    W(\chi_{\delta,V}^{\,r+s})=1.
    \]

    \item[(b)] If $v$ splits in $K/F$, say $v\mathcal O_K=VV^c$ (where $c$ denotes the complex conjugate), then
    \[
    W(\chi_{\delta,V}^{\,r+s})
    W(\chi_{\delta,V^c}^{\,r+s})=1.
    \]
\end{enumerate}
\end{propositioneigo}

\begin{proof}
(a) Suppose first that $v$ is inert in $K/F$. Then $K_V/F_v$ is a quadratic extension and
$\chi_{\delta,V}|_{F_v^\times}=1$. If we choose $\pi_V\notin F_v$ such that $\pi_V^2 \in F_v$, then $K_V = F_v(\pi_V)$ and $\chi_{\delta, V}|_{F_v^\times} = 1$. Then, by \cite[Theorem 3]{frohlich1973functional}, we have
$$
W(\chi_{\delta, V}^{r+s}) = \chi_{\delta, V}^{r+s}(\pi_V) .
$$
Since $\pi_V^2\in F_v^\times$ and $\chi_{\delta,V}$ is trivial on
$F_v^\times$, we have
$\chi_{\delta,V}^{\,r+s}(\pi_V)^2=1$. On the other hand,
$\chi_{\delta,V}$ has $\ell$-power order and $\ell$ is odd. Therefore
$\chi_{\delta,V}^{\,r+s}(\pi_V)=1$, proving (a).

(b) Suppose next that $v$ splits in $K/F$. Then
$K_V\simeq F_v\simeq K_{V^c}$. By the Galois equivariance of the
Hilbert symbol, for $x\in F_v^\times$ we have
\[
\chi_{\delta,V^c}(x)
=
c\bigl(\chi_{\delta,V}(x)\bigr)
=
\chi_{\delta,V}(x)^{-1},
\]
and hence
$\chi_{\delta,V^c}^{\,r+s}
=(\chi_{\delta,V}^{\,r+s})^{-1}$.

By the standard Gauss sum identity \cite[Chapter 7 Exercises 9 (a), p. 301]{ramakrishnan2013fourier}, we obtain
\[
W(\chi_{\delta,V}^{\,r+s})
W((\chi_{\delta,V}^{\,r+s})^{-1})
=
\chi_{\delta,V}^{\,r+s}(-1).
\]
Since $\chi_{\delta,V}^{\,r+s}$ has odd order,
$\chi_{\delta,V}^{\,r+s}(-1)=1$, proving (b).
\end{proof}

Here, instead of directly determining $\phi_V(\pi_V)$, we compute the product $\prod_{V \mid p} \phi_V(\pi_V)$.
First we need the following lemma:
\begin{lemmaeigo}\label{jacobi sum congruence}
	Let $a,b$ be integers such that $\ell^N\nmid a,\ \ell^N\nmid b,\ \ell^N\nmid (a+b)$. Let $\mf{p}\nmid \ell$ be a prime ideal of $K=\Q(\zeta)$. Then,
	\[
	J\left(\leg{\cdot}{\mf{p}}_{\ell^N}^a,\leg{\cdot}{\mf{p}}_{\ell^N}^b\right)\equiv -1\pmod{(1-\zeta)^2}.
	\]
	\end{lemmaeigo}

	\begin{proof}
    Let $\chi$ be a nontrivial multiplicative character and $\psi$ an additive character of the finite field $\F_q=\mc{O}_K/\mf{p}$. The Gauss sum is defined as
	\[
	G(\chi,\psi):=\sum_{a\in\F_q\x} \chi(a)\psi(a).
	\]
        We fix $\psi$ as $\psi(a)=\zeta_p^{\Tr_{\F_q/\F_p}(a)}$ (where $p$ is the characteristic of $\F_q$) from now on. Then the Gauss sum $G(\chi,\psi)$ is simply denoted by $G(\chi)$.

	We set $\lambda=1-\zeta$. We note the well-known relationship between Gauss sums and Jacobi sums:
	\[
	J\left(\leg{\cdot}{\mf{p}}_{\ell^N}^a,\leg{\cdot}{\mf{p}}_{\ell^N}^b\right)
	=G\left(\leg{\cdot}{\mf{p}}_{\ell^N}^a\right)G\left(\leg{\cdot}{\mf{p}}_{\ell^N}^b\right)/G\left(\leg{\cdot}{\mf{p}}_{\ell^N}^{a+b}\right).
	\]

	We have
	\[
	G\left(\leg{\cdot}{\mf{p}}_{\ell^N}\right)
	\equiv \sum_{[a]\in \left(\mathcal{O}_K/\mf{p}\right)\x} 1\cdot \psi(a)
	=-\psi(0)
	=-1\ \pmod{\lambda},
	\]

	and we can write $G\left(\leg{\cdot}{\mf{p}}_{\ell^N}\right)=-1+\lambda z\ (z\in\Z[\zeta,\zeta_p])$. Here, let $R=\Z[\zeta,\zeta_p]/(\lambda^2)$ and $\tilde{\sigma}_a:R\to R$ be the ring homomorphism such that $[\zeta]\mapsto[\zeta]^a,\ [\zeta_p]\mapsto[\zeta_p]$.
    Then,
	\[
	G\left(\leg{\cdot}{\mf{p}}_{\ell^N}^a\right)
	\equiv G\left(\leg{\cdot}{\mf{p}}_{\ell^N}\right)^{\tilde{\sigma}_a}
	\equiv -1+\lambda^{\tilde{\sigma}_a} z^{\tilde{\sigma}_a} \pmod{\lambda^2}.
	\]

	Thus,
	\begin{align}
	    G\left(\leg{\cdot}{\mf{p}}_{\ell^N}^a\right)G\left(\leg{\cdot}{\mf{p}}_{\ell^N}^b\right)
	\equiv& 1-\lambda^{\tilde{\sigma}_a} z^{\tilde{\sigma}_a}-\lambda^{\tilde{\sigma}_b} z^{\tilde{\sigma}_b}\\
	=&1-(1-\zeta^a) z^{\tilde{\sigma}_a}-(1-\zeta^b)z^{\tilde{\sigma}_b}\ \pmod{\lambda^2}.
	\end{align}

	Moreover,
	\[
	G\left(\leg{\cdot}{\mf{p}}_{\ell^N}\right)^{\tilde{\sigma}_{a+b}}
	\equiv -1+\lambda^{\tilde{\sigma}_{a+b}} z^{\tilde{\sigma}_{a+b}}
	\equiv -1+(1-\zeta^{a+b}) z^{\tilde{\sigma}_{a+b}} \pmod{\lambda^2}.
	\]

	Therefore, it suffices to show that $F(z)\equiv 0 \pmod{\lambda^2}$ for all $z\in\Z[\zeta,\zeta_p]$ where $F(z):=(1-\zeta^a) z^{\tilde{\sigma}_a}+(1-\zeta^b)z^{\tilde{\sigma}_b}- (1-\zeta^{a+b}) z^{\tilde{\sigma}_{a+b}}$. 
    For $i\geq 0$, we have $F(\zeta^i)\equiv (1-\zeta^a)+(1-\zeta^b)-(1-\zeta^{a+b})=(1-\zeta^a)(1-\zeta^b)\equiv 0\pmod{\lambda^2}$ and the claim follows since $F(z)$ is $\Z[\zeta_p]$-linear.
	\end{proof}

\begin{propositioneigo}
Let $v\nmid\ell$ be a prime of $F$, let $V$ be a prime of $K$ above $v$, and let $c\in\gal(K/F)$ denote complex conjugation.
\begin{enumerate}
    \item[(a)] If $v$ is inert in $K/F$, then $\phi_V(\pi_V)=-1$.

    \item[(b)] If $v$ splits in $K/F$, then $\phi_V(\pi_V)\phi_{V^c}(\pi_{V^c})=1$.
\end{enumerate}
\end{propositioneigo}

	\begin{proof}
    We first establish the following Galois compatibility for the Jacobi sum: $$(\ast)\,\,\, j_{r,s,t}(\mf{p}_{V})^c=j_{r,s,t}(\mf{p}_{V^c}).$$
    It suffices to verify the corresponding compatibility for the power residue symbol: $\leg{x}{\mf{p}_V}_{\ell^N}^c=\leg{x^c}{\mf{p}_{V^c}}_{\ell^N}$ for $x\in \mc{O}_K\setminus{\mf{p}_V}$.
 Indeed, by the definition of power residue symbol, we have
\[
\leg{x}{\mf{p}_V}_{\ell^N}^c \equiv x^{\frac{q_V-1}{\ell^N}}\equiv \leg{x^c}{\mf{p}_{V^c}}_{\ell^N} \pmod{\mf{p}_{V^c}}.
\]
Thus, the claim follows since both sides are roots of unity,
 
(a) In this case, $\mf{p}_{V^c}=\mf{p}_{V}$. Here, $\Phi_{r,s,t}$ is a representative system of $(\Z/\ell^N\Z)\x/\set{\pm 1}$, thus
$$\left(\mf{p}_{V}^{\Phi_{r,s,t}}\right)^2=\mf{p}_{V}^{\Phi_{r,s,t}+c\Phi_{r,s,t}}=\mf{p}_{V}^{\sum_{\sigma\in\gal(K/\Q)}\sigma}=\prod_{[\sigma]\in\gal(K/\Q)/D_V}{\mf{p}_{V}^{\sigma}}^{\# D_V}=(\mc{N}\mf{p}_{V})=(q_v)^2$$
	where $\mc{N}$ denotes the absolute norm and $q_v:=\mc{N}\mf{p}_{v}$. Thus, by the uniqueness of the prime ideal factorization, we have $(j_{r,s,t}(\mf{p}_{V}))=\mf{p}_{V}^{\Phi_{r,s,t}}=(q_{v})$, therefore, we can write $j_{r,s,t}(\mf{p}_{V})=uq_v\ (u\in\mc{O}_K\x)$. Since the absolute value of $u$ is $1$ for any complex conjugate, by Kronecker's theorem, it is a root of unity. Now, since $v$ is inert, we know that $\frob_{v}=c$ which implies $q_v\equiv -1\pmod{\ell^N}$, and the Jacobi sum $j_{r,s,t}(\mf{p}_{V})$ is real by $(\ast)$. Therefore, we have $u=-1$ by \Cref{jacobi sum congruence} and
	$\phi_{V}(\pi_V)=j_{r,s,t}
    (\mf{p}_{V})\abs{\pi_V}_V^{\frac{1}{2}}=-1.$
    
	(b) Since $\phi_V$ is unitary, we only have to show \(\phi_V(\pi_V)^c=\phi_{V^c}(\pi_{V^c}) \), but this immediately follows from $(\ast)$.
\end{proof}

\begin{propositioneigo}
    \[
    \prod_{V \mid p} \phi_V(\pi_V) = \leg{p}{\ell}.
    \]
\end{propositioneigo}
\begin{proof}
    Let $f, f'$, and  $f_1$ be the relative degrees of \(p\) in $K, F$, and $\Q(\zeta_{\ell})$ respectively. Similarly let $g, g'$, and $g_1$ be the number of prime ideals lying above \(p\) in these respective fields.
    Since \(g_1 \mid g\) and \(f = f_1 \cdot \text{(power of \(\ell\))}\), \(g_1\) and \(g\) have the same parity. Furthermore, since both \(K/\Q\) and \(F/\Q\) are Galois, the prime ideals of \(F\) lying above \(p\) either (a) all split or (b) all remain inert.

    \textbf{Case \(\leg{p}{\ell} = 1\).}

    Since \(1 = \leg{p}{\ell} \equiv p^{\frac{\ell-1}{2}} \pmod{\ell}\), it follows that \(f_1 \mid \frac{\ell-1}{2}\). Hence, \(2 \mid g_1\), therefore, \(2 \mid g\).
    In case (a), \(\prod_{V \mid p} \phi_V(\pi_V) = 1^{g'} = 1\).
    In case (b), since \(g' = g\), we have \(\prod_{V \mid p} \phi_V(\pi_V) = (-1)^{g'} = 1\).

    \textbf{Case \(\leg{p}{\ell} = -1\).}
    
    Since \(2 \nmid g\), only (b) can occur. In this case, \(\prod_{V \mid p} \phi_V(\pi_V) = (-1)^g = -1\).
\end{proof}

From the above, we obtain the following:

\begin{propositioneigo}
    \[
    W_p(\phi_{\delta}) =
    \begin{dcases}
        i^{-\frac{\ell^{N-1}(\ell-1)}{2}} & (p = \infty) \\
        1 & (p \nmid \delta) \\
        \leg{p}{\ell} & (p \mid \delta).
    \end{dcases}
    \]
\end{propositioneigo}

\subsection{\ Calculation of $W_{\ell}(\phi_{\delta})$}

Let $\eta$ be the character of $\A_F^\times$ corresponding to the extension $K/F$ via class field theory. Then, by \cite[Proposition 1]{rohrlich1982}, we have $\phi_{\delta}|_{\A_F} = \eta$. Thus, it suffices to compute $W(\eta_{\pi_F})$ and the relative root number $W(\phi_{\delta,\pi}, \eta_{\pi_F}) := W(\phi_{\delta,\pi}) / W(\eta_{\pi_F})$.

\subsubsection{\ Calculation of $W(\eta_{\pi_F})$.}\label{eta}

From \cite[Theorem 1]{frohlich1973functional}, we know that $W(\eta) = 1$, thus we have
\[
W(\eta_{\pi_F}) = \left(\prod_{V \mid \infty} (-i)^{\abs{-1}} \right)^{-1} = i^{\frac{\ell^{N-1}(\ell-1)}{2}}.
\]

\subsubsection{Preparation for some congruences for Fleck numbers}

Let $p$ be a prime number through this subsection.

We recall that alternating binomial sums of the following form
$$C_n \binom{M}{r}=\sum_{\substack{k \equiv r\pmod{n}}}(-1)^k\binom{M}{k} $$
where $M\geq 0,n>0$ are integers and $r$ is an integer (Notice that $k$ runs in fact among $0\leq k\leq M$ by the notation of binomial coefficients and this is a finite sum) are called \textit{Fleck numbers} (\cite[Introduction (1.3)]{lettington2011fleck}), and they are of interest to have some nontrivial supercongruences (See \cite[Introduction]{Sun2006}). We note that we can put a rational number into $r$ when it is well-defined in mod $n$.

We need some congruences for Fleck numbers to calculate the Hilbert symbol
required in the next subsection. The following is a classical result that will
be used later.

\begin{propositioneigo}(Weisman's formula \cite[Theorem 13]{weisman1977some} (\cite[(1.2)]{Sun2006}))\label{weisman}
        Let $r$ be an integer, $s$ be a positive integer.
    
    (a) Let $M$ be an integer such that $M>p^{s-1}$. Then, the following holds:
    $$C_{p^s} \binom{M}{r} \equiv 0 \pmod{p^{\left\lfloor\frac{M-p^{s-1}}{\varphi\left(p^s\right)}\right\rfloor}}.$$

    (b) Let $m$ be a positive integer. Then, the following holds:
    $$C_{p^s} \binom{p^{s-1}(pm-m+1)-1}{r}\equiv (-p)^{m-1}\pmod{p^m}.$$
\end{propositioneigo}

From now on, let \(p\) be odd.

We prove another type of congruences for Fleck numbers:

\begin{theoremeigo}\label{fleck}
Let \(n\ge m\ge 1\) be integers. Let \(d\) be a positive
odd integer. 
Consider the following Fleck number
\[
  S_n(d):=C_{p^n} \binom{(2p^m-1)d-1}{\frac{(2p^m-1)d}{2}}=
  \sum_{k\equiv \frac{(2p^m-1)d}{2}\;(\operatorname{mod} p^n)}
  (-1)^k
  \binom{(2p^m-1)d-1}{k}.
\]
Then
\[
  S_n(d)\equiv
  \left\{
  \begin{array}{lll}
    -p^m & \pmod {p^{m+1}} & (d=p^{n-m})\\[2mm]
    0    & \pmod {p^{m+1}} & (d\ne p^{n-m}).
  \end{array}
  \right.
\]
\end{theoremeigo}

\begin{proof}
Let $0\le r_d<p^n$ be the integer determined by
\[
  2r_d\equiv (2p^m-1)d\pmod {p^n},
\]
and \(j_d\) be defined by
\[
  (2p^m-1)d=2r_d+j_d p^n,
\]
noting that \(j_d\) is odd. We may assume \(j_d>0\).

Using $\binom{M-1}{k}=\frac{M-k}{M}\binom{M}{k}$ and writing $k=r_d+jp^n\ (0\le j\le j_d)$, we obtain
\begin{align}
    (2p^m-1)dS_n(d)=
  \sum_{\substack{k\ge0\\ k\equiv r_d\;(\operatorname{mod} p^n)}}
  (-1)^k \{(2p^m-1)d-k\}
  \binom{(2p^m-1)d}{k}\\
  =
  (-1)^{r_d}\sum_{0\le j\le j_d}
  (-1)^j \{r_d+(j_d-j)p^n\}
  \binom{(2p^m-1)d}{r_d+jp^n}.
\end{align}
Denoting the \(j\)-th summand of $\dsum_{0\le j\le j_d}$ by $T_j$, we have $T_j+T_{j_d-j}=(-1)^j \cdot(j_d-2j)p^n
  \binom{(2p^m-1)d}{r_d+jp^n}$. Thus, pairing the terms
with indices \(j\) and \(j_d-j\) gives
\[
  2(-1)^{r_d}(2p^m-1)d S_n(d)
  =
  p^n
  \sum_{0\le j\le j_d}
  (-1)^j(j_d-2j)
  \binom{(2p^m-1)d}{r_d+jp^n}.
\]

Then, the case $p^{n-m}\nmid d$ immediately follows from this and we can write $d=p^{n-m}h$.

\textbf{Case $h>1$.}
Let $$W_h:=\sum_{0\le j\le j_d}
  (-1)^j(j_d-2j)
  \binom{(2p^m-1)d}{r_d+jp^n}$$ and it remains to show $\ord_p(W_h)\ge \ord_p(h)+1$.
  By \cite[Theorem 1.1]{sun--davis2007combinatorial}, we have 
  $$\ord_p(W_h)\ge \floor{\frac{(2p^m-1)d}{p^n}}-1\ge \left(h-\ceil{\frac{h}{p^m}}\right)+(h-1)$$
  and the claim follows because $h\ge p^e\ge 3^e\ge e+2$ when $e=\ord_p(h)\ge 1$ and $e=0$ is trivial.

\textbf{Case $h=1$.}
 In this case, $r_d=\frac{p^n-p^{n-m}}{2}$ and $j_d=1$. So, 
 $$W_1=\sum_{0\le j\le 1}
  (-1)^j(j_d-2j)
  \binom{(2p^m-1)p^{n-m}}{r_d+jp^n}=2\binom{2p^n-p^{n-m}}{r_d}.$$
  From the $p$-adic expansions
  $$(2p^m-1)p^{n-m}=(p-1)\cdot p^{n-m}+\cdots +(p-1)\cdot p^{n-1}+p^n,\ r_d=\frac{p-1}{2}\cdot p^{n-m}+\cdots +\frac{p-1}{2}\cdot p^{n-1},$$ we have
  $$\binom{(2p^m-1)p^{n-m}}{r_d}\equiv \binom{p-1}{\frac{p-1}{2}}^m\equiv (-1)^{\frac{p-1}{2}m}=(-1)^{r_d}\pmod{p}$$
  by Lucas' theorem. Hence, we have $-S_n(d)/p^m\equiv 1\pmod{p}$ and this completes the proof.
\end{proof}

\begin{propositioneigo}\label{fleck2}
Let \(p\) be an odd prime, $n\geq 0$ and $k\geq 0$ be integers, and let
$0\leq v<p^n$. Then
\[
C_{p^n}\binom{p^n+2k-1}{v}
\equiv
\begin{cases}
1\pmod{p} & (k=0)\\
0\pmod{p} & (k>0).
\end{cases}
\]
\end{propositioneigo}

\begin{proof}
Suppose first that $n=0$. Then $v=0$, and the congruence condition
modulo $p^n=1$ is vacuous. Hence
\[
C_{p^n}\binom{p^n+2k-1}{0}
=
\sum_{k'=0}^{2k}
(-1)^{k'}\binom{2k}{k'}.
\]
This is equal to $1$ if $k=0$ and to $0$ if $k>0$ by the binomial
theorem.

Now suppose that $n\geq1$. The case $k>0$ follows from Weisman's formula \cite{weisman1977some}. For $k=0$, write
\[
v=\sum_{j=0}^{n-1}a_jp^j
\]
for the $p$-adic expansion of $v$. By Lucas' theorem,
\[
\binom{p^n-1}{v}
\equiv
\prod_{j=0}^{n-1}\binom{p-1}{a_j}
\equiv
(-1)^{\sum_{j=0}^{n-1}a_j}
\equiv
(-1)^{v}
\pmod{p}.
\]
and the claim follows. We note that for every integer $1\leq a\leq p-1$, one has $$\binom{p-1}{a}\equiv -\binom{p-1}{a-1}\equiv (-1)^a \pmod{p}$$
by the Pascal identity.
\end{proof}

Here, we introduce a function $J(n,f)$ involving Fleck numbers:
\begin{propositioneigo}\label{J}
    For a positive integer $n\geq 1$ and an integer $f>1$, we define the function $J(n,f)$ as follows:

    For integers $u\in\set{0,1}$ and $i\geq 0$, let $M_{i}=(f-1)(i+1)-1$ and $$U^{(u)}_i:=C_{p^{n-u}}\binom{M_i}{\frac{M_i+1}{2}}=\dsum_{k\equiv \frac{M_i+1}{2} \pmod{p^{n-u}}} (-1)^k \binom{M_i}{k}.$$
    Then, the sum
	$$J(n,f):=\sum_{i\geq 0\atop\text{$i$:even}}
\left(p U^{(0)}_i-U^{(1)}_i\right)$$
    converges $p$-adically and $J(n,f)\in\Z_{p}$.
\end{propositioneigo}
\begin{proof}
 When $n=u=1$, $U_i^{(1)}=0$ for $i>0$ and this converges. Assume $n-u\geq 1$. By Weisman's formula (\Cref{weisman}), 
$$\ord_{p}(U^{(u)}_i)\geq\floor{\frac{M_{i}-p^{n-u-1}}{\varphi(p^{n-u})}}.$$
Because the right hand tends to $\infty$ when $i\to\infty$, the claim follows.
\end{proof}

\begin{remarkeigo}
    We can explicitly describe a truncated finite sum $\tilde{J}(n,f)\in\Z$ which represents $J(n,f) \pmod{p^n}$.
    In fact, by Weisman's formula (\Cref{weisman}), it is a sufficient condition for $U^{(u)}_{j+2}\equiv 0\pmod{p^{n+u-1}}$ to hold that 
$$\floor{\frac{M_{j+2}-p^{n-u-1}}{\varphi(p^{n-u})}}\geq n+u-1$$
when $n-u>0$.
Thus, $U^{(u)}_{j+2}\equiv 0\pmod{p^{n+u-1}}$ holds for all $j\geq j_u$ where
\begin{align}
    j_u:=&\min\Set{j:\text{even}}{j\geq \dfrac{p^{n-u-1}(np+(u-1)p-n+2-u)+1}{f-1} - 3}\\
=&2\ceil{\frac{1}{2}\left(\dfrac{p^{n-u-1}(np+(u-1)p-n+2-u)+1}{f-1} - 3\right)},
\end{align}
noticing that the case $n-u=0$ is also covered by this formula.
Therefore, the following integer
    \begin{align}
		\tilde{J}(n,f):=p\sum_{0\leq i\leq j_0\atop\text{$i$:even}}U^{(0)}_i-\sum_{0\leq i\leq j_1\atop\text{$i$:even}}U^{(1)}_i
	\end{align}
    satisfies
    $J(n,f)\equiv \tilde{J}(n,f) \pmod{p^n}.$
\end{remarkeigo}

\begin{theoremeigo}\label{thm:conjecture-6-1}
Let $1\leq w\leq N$, and $i\geq 0$ be an even integer, and set $n=N$ and $f=2p^{N-w}$ in the above definition.
If $1\leq w<N$, then, for $u\in\{0,1\}$ 
\[
U_i^{(u)}
\equiv
\begin{cases}
-p^{N-w}\pmod{p^{N-w+1}} & (i=p^{w-u}-1)\\
0\pmod{p^{N-w+1}} & (i\neq p^{w-u}-1).
\end{cases}
\]
If $w=N$, then
\[
U_i^{(u)}
\equiv
\begin{cases}
1\pmod{p} & (i=p^{N-u}-1)\\
0\pmod{p} & (i\neq p^{N-u}-1).
\end{cases}
\]
Consequently,
\[
\frac{J(N,2p^{N-w})}{p^{N-w}}
\equiv
\begin{cases}
1\pmod{p} & (1\leq w<N)\\
-1\pmod{p} & (w=N).
\end{cases}
\]
\end{theoremeigo}

\begin{proof}
Suppose first that $1\leq w<N$ and put $m=N-w$. Since $M_i=(2p^m-1)(i+1)-1$,
$U_i^{(u)}$ is the sum $S_{N-u}(i+1)$ in \Cref{fleck}. Hence
\[
U_i^{(u)}
\equiv
\begin{cases}
-p^m\pmod{p^{m+1}} & (i+1=p^{N-u-m})\\
0\pmod{p^{m+1}} & (i+1\neq p^{N-u-m}).
\end{cases}
\]

Therefore,
\[
J(N,2p^{N-w})
=\sum_{\substack{i\geq0\\i:\mathrm{even}}}
(p U_i^{(0)}-U_i^{(1)})
\equiv p^{N-w}\pmod{p^{N-w+1}},
\]
and the first case follows.

Suppose next that $w=N$. Then $f=2$ and $M_i=i$. By \Cref{fleck2},
with $n=N-u$,
\[
U_i^{(u)}
\equiv
\begin{cases}
1\pmod{p} & (i=p^{N-u}-1)\\
0\pmod{p} & (i\neq p^{N-u}-1),
\end{cases}
\]
noticing that the sum is empty when $i+1<p^{N-u}$.
Hence
\[
J(N,2)
=\sum_{\substack{i\geq0\\i:\mathrm{even}}}
(p U_i^{(0)}-U_i^{(1)})
\equiv-1\pmod{p}.
\]
Since $p^{N-w}=1$, the result follows.
\end{proof}

\subsubsection{\ Calculation of the relative root number $W(\phi_{\delta,\pi}, \eta_{\pi_F})$.}

Let $f_{\delta,\pi}$ be the conductor exponent of $\chi_{\delta,\pi}$.
We compute the \textit{relative root number} $W(\phi_{\delta,\pi}, \eta_{\pi_F})$ by using Rohrlich's formula \cite[Proposition 2]{rohrlich1992root} (We also refer to this for relative root numbers). To do it, we first have to calculate $\chi_{\delta,\pi}(1 + \pi^{f_{\delta,\pi}-1})$.

Let $\Delta \geq 1$ be an $\ell^N$-th power-free integer, satisfying either $\ord_{\ell}(\Delta) = 0$ or $\ell \nmid \ord_{\ell}(\Delta)$.

We express $\Delta$ as $\Delta = \epsilon \ell^b(1 + c)$ where $\epsilon \in \mu_{\ell-1}(\Q_\ell)$, $b \in \Z$, and $c \in \ell \Z_\ell$ and let
\[
w(\Delta) := \min\{\ord_\ell(b), \ord_\ell(c)\}.
\]
We write $f(\Delta)$ for the conductor exponent of the Hilbert symbol $(\cdot, \Delta)_{\ell^N}$.

\begin{propositioneigo}\label{conductor}
    For $\Delta\in\Q_\ell^\times$, write
$\Delta=\epsilon\ell^b(1+c)$, where $\epsilon \in \mu_{\ell-1}(\Q_\ell)$, $b \in \Z$, and $c \in \ell \Z_\ell$, and set $w = w(\Delta)$. Then,
    \[
    f(\Delta) =
    \begin{dcases}
        \ell^{N-1}(\ell+1) & (w = 0) \\
        2\ell^{N-w} & (1 \leq w < N,\ \ord_\ell(b + c) = w) \\
        \ell^{N-w-1}(\ell+1) & (1 \leq w < N,\ \ord_\ell(b + c) > w) \\
        2 & (w = N = \ord_\ell(c)) \\
        0 & (\text{$w > N$ or $w = N \neq \ord_\ell(c)$}).
    \end{dcases}
    \]
\end{propositioneigo}
\begin{proof}
    Since $\ord_\pi(1 - \zeta_{\ell^i}) = \ell^{N-i}$ for $1 \leq i \leq N$, the result follows from \cite[Theorem 8]{sharifi2001norm}.
\end{proof}

We note that $f_{\Delta,\pi} = f(\Delta)$.

Let $\Delta = \epsilon \ell^b(1 + c)$ with $\epsilon \in \mu_{\ell-1}(\Q_\ell)$, $0 \leq b$, and $c \in \ell \Z_\ell$. We write $w = w(\Delta) = \min\{\ord_\ell(b), \ord_\ell(c)\}$. Then, we have
\[
w =
\begin{cases}
    \ord_\ell(c) & (b = 0) \\
    0 & (b \neq 0)
\end{cases}
\]
since $\ord_{\ell}(\Delta) = 0$ or $\ell \nmid \ord_{\ell}(\Delta)$. 

\begin{lemmaeigo}\label{trace mod}
    For integer $t\geq 0$,
    $\Tr_{\Q_{\ell}(\zeta)/\Q_{\ell}}({\pi'}^m a)\equiv 0\pmod{\ell^{N+t}}$ (for all $a \in \Z_{\ell}[\zeta]$) is equivalent to $m \geq \ell^{N-1}(\ell t-t+1)$.
\end{lemmaeigo}
\begin{proof}
    The different of $\Q_{\ell}(\zeta)/\Q_{\ell}$ is $\mf{d}_{\Q_{\ell}(\zeta)/\Q_{\ell}} = (\pi')^d$ where $d = \ell^{N-1}(\ell N-N-1)$.
    
Thus, we have the lemma by \cite[Chapter 3, Section 3, Proposition 7]{serre2013local} with $\mf{b}=(\pi'),\ \mf{a}=(\ell^{N+t})$.
\end{proof}

In the following, we calculate $\chi_{\Delta,\pi}(1 + \pi^{f-1})$. First, we calculate the following Hilbert symbol:

\begin{propositioneigo}\label{1+c}
Let \(f>1\) be an integer, \(c\in\ell\mathbf Z_\ell\setminus{\set{0}}\), and $w=\ord_{\ell}(c)$.
Then
\[
[1+\pi^{f-1},1+c]_{\ell^N}
\equiv
\begin{cases}
\displaystyle
2(f-1)J(N,f)
\frac{1}{\ell}\log(1+c)
\pmod{\ell^N}
& (\text{\(f\) is even, $w\leq N$ and $f\leq 2\ell^{N-w}$})\\[10pt]
0\pmod{\ell^N}
& (\text{otherwise})
\end{cases}
\]
where $\log$ is the $\ell$-adic logarithm and the function $J(n,f)$ is defined in \Cref{J}.
Moreover, the term $\log(1+c)$ can be replaced simply by the term $c$ when $f>2\ell^{N-2w}$.
\end{propositioneigo}
\begin{proof}
First, since the conductor exponent of $(\cdot ,1+c)_{\ell^N}$ is $2\ell^{N-w}$ for $w\leq N$ and $0$ for $w>N$ by \Cref{conductor}, we assume $w\leq N$ from now on, and the value of Hilbert symbol is $1$ when $f> 2\ell^{N-w}$.
Since $\ord_{\ell}(c) \geq 1$ and $\ell > 2$, we have
$$\ord_{\pi}(c) = \ord_{\ell}(c)\ell^{N-1}(\ell-1) \geq 2\ell^{N-1}.$$

Therefore, the Artin-Hasse-Iwasawa formula \cite[Chapter V, Proposition 3.8]{neukirch2013algebraic} can be applied, and we have
\begin{flalign}
	&[1+\pi^{f-1},1+c]_{\ell^N}&\\
	&\equiv -\frac{1}{\ell^N}\Tr_{\Q_{\ell}(\zeta)/\Q_{\ell}}\left(\zeta\log(1+c)D\log (1+\pi^{f-1})\right)&\\
	&\equiv -\frac{1}{\ell^N}\log(1+c)\Tr_{\Q_{\ell}(\zeta)/\Q_{\ell}}\left((1-\pi')\frac{(f-1)\pi^{f-2}}{1+\pi^{f-1}}\frac{d\pi}{d\pi'}\right)&\\
	&\equiv -\frac{f-1}{\ell}\frac{1}{\ell^{N-1}}\Tr_{\Q_{\ell}(\zeta)/\Q_{\ell}}\left((1-\pi')\frac{\pi^{f-2}}{1+\pi^{f-1}}(-2+\frac{d(\pi'{\pi'}^c)}{d\pi'})\right)\log (1+c) \pmod{\ell^{N}}.
\end{flalign}
The second line from the bottom follows from $\pi=-2\pi'+\pi'{\pi'}^c$.
Here, we can evaluate the inside of the trace modulo $\pi'^M$ where $M := \ell^{N-1}(\ell N-N-\ell+2)$ by \Cref{trace mod}:
\begin{align}
    \frac{d(\pi'{\pi'}^c)}{d\pi'}=\frac{d}{d\pi'}\left(-\frac{{\pi'}^2}{1-\pi'}\right)
	\equiv & -(2\pi' + 3{\pi'}^2 + 4{\pi'}^3 + \cdots + M{\pi'}^{M-1}) \\
	= & -\left(\frac{M{\pi'}^{M+1} - (M+1){\pi'}^M + 1}{(1-\pi')^2} - 1\right) \\
	\equiv & -\frac{1}{(1-\pi')^2} + 1 \pmod{{\pi'}^M}.
\end{align}
Thus, we have
\begin{align}
	&[1+\pi^{f-1},1+c]_{\ell^N} \\
	\equiv &\frac{f-1}{\ell}\cdot \frac{1}{\ell^{N-1}}
	\Tr_{\Q_{\ell}(\zeta)/\Q_{\ell}}\left(\frac{\pi^{f-2}}{1+\pi^{f-1}}(2-\pi'{\pi'}^c )\right)\log (1+c)  \pmod{\ell^N}.
\end{align}

In the $\ell$-adic topology, we have
\begin{align}
	\frac{\pi^{f-2}}{1+\pi^{f-1}}(2-\pi'{\pi'}^c)
	= & \frac{\pi^{f-2}}{1+\pi^{f-1}}(\zeta +\zeta^{-1}) \\
	= & \sum_{i \geq 0} (-1)^i \pi^{(f-1)i + f-2} (\zeta + \zeta^{-1}).
\end{align}
Therefore, we need to compute the trace of $S_i := \pi^{M_i} (\zeta + \zeta^{-1})$ where $M_i:=(f-1)(i+1) -1$.
Let $P_n = \zeta^n + \zeta^{-n}$ and $Q_n = \zeta^n - \zeta^{-n}$ for $n \in \Z$ and we have for every integer \(n\geq0\), 
\[
\pi^n(\zeta+\zeta^{-1})
=
\begin{cases}
\displaystyle
\sum_{k=0}^{n}
(-1)^k\binom{n}{k}P_{n+1-2k}
& (\text{\(n\) is even})\\[10pt]
\displaystyle
\sum_{k=0}^{n}
(-1)^k\binom{n}{k}Q_{n+1-2k}
& (\text{\(n\) is odd}).
\end{cases}
\]
	This follows immediately from a binomial expansion $(-1)^n \zeta\inv \pi^n=\dsum_{0\leq k\leq n} (-1)^k\binom{n}{k}\zeta^{-(n+1-2k)}$.

We note again that
$
\Tr_{\Q_{\ell}(\zeta)/\Q_{\ell}}(\zeta^n) 
=\ell^N \1_{\ell^N \mid n} 
	-\ell^{N-1} \1_{\ell^{N-1} | n}
$
for $n\in\Z$, and therefore, $\Tr_{\Q_{\ell}(\zeta)/\Q_{\ell}}(P_n) = 2\Tr_{\Q_{\ell}(\zeta)/\Q_{\ell}}(\zeta^n)$, while $\Tr_{\Q_{\ell}(\zeta)/\Q_{\ell}}(Q_n) = 0$. Thus, $\Tr_{\Q_{\ell}(\zeta)/\Q_{\ell}}(S_i) = 0$ when $f$ is odd. Now, we assume $f$ is even. In this case, we have

\begin{align}
    &\frac{1}{\ell^{N-1}}
	\Tr_{\Q_{\ell}(\zeta)/\Q_{\ell}}\left(\frac{\pi^{f-2}}{1+\pi^{f-1}}(2-\pi'{\pi'}^c )\right)\\
    =&\frac{1}{\ell^{N-1}}
\sum_{i\geq 0}(-1)^i\Tr_{\Q_{\ell}(\zeta)/\Q_{\ell}}\left(S_i\right)\\
=&\frac{2}{\ell^{N-1}}
\sum_{i\geq 0\atop\text{$i$:even}}
\sum_{k=0}^{M_i} (-1)^k \binom{M_i}{k}
\Tr_{\Q_{\ell}(\zeta)/\Q_{\ell}}(\zeta^{M_i+1-2k})\\
=&2\sum_{i\geq 0\atop\text{$i$:even}}
\left(\ell U^{(0)}_i-U^{(1)}_i\right)\\
=&2J(N,f)
\end{align}
where $U^{(u)}_i:=\dsum_{0\leq k\leq M_i\atop \ell^{N-u}\mid M_i+1-2k} (-1)^k \binom{M_i}{k}$ for $u\in\set{0,1}$, and each of these sums corresponds to the one in \Cref{J}.
Thus, the first assertion follows.

We consider the last assertion. First, we have $$\ord_{\ell}(\log(1+c)-c)\geq 2\ord_{\ell}(c)$$
since $$\ord_{\ell}\left(\frac{c^k}{k}\right)\geq 2\ord_{\ell}(c)\ (k\geq 2).$$ Indeed, write $t:=\ord_\ell(k)$ for simplicity, then $k\geq \ell^t \geq 3^t\geq t+2$ when $t>0$, and the case for $t=0$ is trivial.
Thus, the contribution of \(\log(1+c)-c\) vanishes modulo \(\ell^N\) whenever $\ord_{\ell}((f-1)J(N,f))>N-2w$.

Next, we consider the character $(\cdot ,1+c)_{\ell^N}^{\ell^w}$. Write $(1+c)^{\ell^w}=1+c_w\ (c_w\in\Z_{\ell})$, and $\ord_{\ell}(c_w)=2w$. Therefore, the conductor exponent of the character is less than or equal to $2\ell^{N-2w}$, and for integer $f>2\ell^{N-2w}$, we have $(1+\pi^{f-1} ,1+c)_{\ell^N}^{\ell^w}=1$ implying 
$$\ell^w\cdot 2(f-1)J(N,f)
\frac{1}{\ell}\log(1+c)
\equiv 0\pmod{\ell^N}.$$
This is equivalent to $\ord_{\ell}((f-1)J(N,f))>N-2w$, and the assertion follows.
\end{proof}

Now, we can compute $(1 + \pi^{f_{\Delta,\pi}-1}, \Delta)_{\ell^N}$ when $(\cdot, \Delta)_{K_{\pi}, \ell^N}$ ramifies (i.e. $f_{\Delta,\pi} > 0$).

\begin{propositioneigo}\label{value of hilbert symbol}
	\[
	(1 + \pi^{f_{\Delta,\pi}-1}, \Delta)_{\ell^N}  =
	\begin{dcases}
    \zeta_{\ell}^{2b} & (b \neq 0) \\
		\zeta_{\ell}^{-2c_0} & (b = 0,\ 1\leq \ord_{\ell}(c)\leq N)
	\end{dcases}
	\]
    where $c_0:=\dfrac{c}{\ell^{\ord_{\ell}(c)}}$. (recall that $b=\ord_{\ell}(\Delta)$.)
    \mathtoolsset{showonlyrefs=false}
	\end{propositioneigo}
\mathtoolsset{showonlyrefs=true}
\begin{proof}
    Write
	$$(1+\pi^{f_{\Delta,\pi}-1},\Delta)_{\ell^N}
	=(1+\pi^{f_{\Delta,\pi}-1},\epsilon)_{\ell^N}(1+\pi^{f_{\Delta,\pi}-1},\ell)_{\ell^N}^b (1+\pi^{f_{\Delta,\pi}-1},1+c)_{\ell^N}.$$
    We first see $(1+\pi^{f_{\Delta,\pi}-1},\epsilon)_{\ell^N}=1$ since $\epsilon^{\ell^N}=\epsilon$, and we need to calculate the remaining two factors. We note that it is sufficient to compute $(1+\pi^{f_{\Delta,\pi}-1},\ell)_{\ell^N}$ when $b\neq 0$.
    
    \begin{itemize}[leftmargin=0pt]
        \item \textbf{Calculation of $(1+\pi^{f_{\Delta,\pi}-1},\ell)_{\ell^N}$ (when $b\neq 0$).}
        
        $w=0$ since $b\neq 0$, and $f_{\Delta,\pi}=\ell^{N-1}(\ell+1)$ by \Cref{conductor}. Now, $f_{\Delta,\pi}=f(\ell)$ holds, and we have
	$$1+\pi^{f_{\Delta,\pi}-1}\equiv 1+(-2\pi'+\pi'{\pi'}^c)^{f_{\Delta,\pi}-1}
	\equiv 1+(-2\pi')^{f_{\Delta,\pi}-1}
	\equiv (1+{\pi'}^{f_{\Delta,\pi}-1})^{-2} \pmod{{\pi'}^{f(\ell)}}.$$
    Thus, $$[1+\pi^{f_{\Delta,\pi}-1},\ell]_{\ell^N}=-2[1+{\pi'}^{f_{\Delta,\pi}-1},\ell]_{\ell^N}.$$
    
	Let $$F(X):=1+X^{f_{\Delta,\pi}-1},\ G(X):=\dprod_{\substack{1\leq i\leq \ell^N-1\\ \ell\nmid i}} (1-(1-X)^i),$$ note that $G(\pi')=\ell$, and let
	$$\widetilde{\Lambda}(X):=\log(1+X^{f_{\Delta,\pi}-1})-\frac{1}{\ell}\log(1+(1-(1-X)^{\ell})^{f_{\Delta,\pi}-1}).$$
	By the equation on \cite[p. 89, line 2]{coleman1988stable},
	\begin{align}
		&[1+{\pi'}^{f_{\Delta,\pi}-1},\ell]_{\ell^N}\\
		=&-F'(0)\log (N_{\Q_{\ell}(\zeta)/\Q_{\ell}}(G({\pi'})))-\frac{1}{\ell^N}\Tr_{\Q_{\ell}(\zeta)/\Q_{\ell}} \left(\widetilde{\Lambda}(1-\zeta)\frac{DG}{G}(1-\zeta)\right)\ \pmod{\ell^N}
	\end{align}
    assuming $\log(\ell)=0$ (i.e. the $\log$ is the Iwasawa logarithm). Be careful that the definition of a Hilbert symbol used in \cite[p. 43]{coleman1988stable} is the reciprocal one of the usual definition. 
    
    We immediately have $F'(0)\log (N_{\Q_{\ell}(\zeta)/\Q_{\ell}}(G({\pi'})))=0.$ Moreover, \cite[p. 89, Theorem 6.3]{coleman1988stable} allows us to see $\widetilde{\Lambda}$ in modulo $X^{f(\ell)}\Z_{\ell}[[X]]$.
    Since
	\begin{align}
		\widetilde{\Lambda}(X)
		\equiv X^{f(\ell)-1}-\frac{1}{\ell}(\ell X)^{f(\ell)-1}
		= (1-\ell^{f(\ell)-2})X^{f(\ell)-1}\ \pmod{X^{f(\ell)}\Z_{\ell}[[X]]},
	\end{align}

	we have
	\begin{align}
		[1+{\pi'}^{f_{\Delta,\pi}-1},\ell]_{\ell^N}
		=&-\frac{1}{\ell^N}\Tr_{\Q_{\ell}(\zeta)/\Q_{\ell}} \left((1-\ell^{f(\ell)-2}){\pi'}^{f(\ell)-1}\frac{DG}{G}(\pi')\right)\ \pmod{\ell^N}\\
		=&-\frac{1}{\ell^N}\Tr_{\Q_{\ell}(\zeta)/\Q_{\ell}} \left({\pi'}^{f(\ell)-1}\frac{DG}{G}(\pi')\right)\ \pmod{\ell^N}.
	\end{align}

    Now, we evaluate the $\pi'$-adic order of $\dfrac{DG}{G}(\pi')$.
    First,
    \begin{align}
        \ell
        &={\pi'}^{\varphi(\ell^N)}\dprod_{1\leq i\leq \ell^N-1\atop \ell\nmid i} \frac{1-\zeta^i}{1-\zeta}\\
        &\equiv {\pi'}^{\varphi(\ell^N)}\dprod_{1\leq i\leq \ell^N-1\atop \ell\nmid i} i
        \equiv {\pi'}^{\varphi(\ell^N)}\dprod_{1\leq i\leq \ell^N-1\atop \ell\nmid i} ((\ell-1)!)^{\ell^{N-1}}
        \equiv -{\pi'}^{\varphi(\ell^N)}\pmod{{\pi'}^{\varphi(\ell^N)+1}},
    \end{align}
    and $\ell\in -{\pi'}^{\varphi(\ell^N)}U^{(1)}$ where $U^{(1)}$ denotes the group of principal units.
    
    Second,
    $\pi_1:=1-\zeta_{\ell}=1-(1-\pi')^{\ell^{N-1}}\equiv {\pi'}^{\ell^{N-1}}\pmod{{\pi'}^{\ell^{N-1}+1}}$, and $\pi_1\in {\pi'}^{\ell^{N-1}}U^{(1)}$.
    
    Thus, we have
    $$\frac{DG}{G}(\pi')\in \frac{\ell^{N-1}}{\pi_1}U^{(1)}=\frac{(-{\pi'}^{\varphi(\ell^N)})^{N-1}}{{\pi'}^{\ell^{N-1}}}U^{(1)}=(-1)^{N-1}\pi'^{\ell^{N-1}(\ell N-\ell-N)}U^{(1)}$$
    by \cite[Lemma 6.5 (i)]{coleman1988stable} (we note that there is a sign error in this lemma. Indeed, the formula for $D[a]/[a]$ on p. 91 has a sign error; the true expression is $D[a]/[a]=-\frac{a(1-x)^a}{1-(1-x)^a}$).
    
		By the way, let $M:=f(\ell)+\ell^{N-1}(\ell N-\ell-N)-1=\ell^{N-1}(\ell N-N+1)-1$, then since we can evaluate the inside of the trace modulo $\pi'^{M+1}$ by \Cref{trace mod}, all of the principal units $u\in U^{(1)}$ satisfy
		$$\frac{1}{\ell^N}\Tr_{\Q_{\ell}(\zeta)/\Q_{\ell}} \left({\pi'}^{M+1}u\right)\equiv 0\ \pmod{\ell^N},$$
        and we have
        \begin{align}
		    [1+{\pi'}^{f_{\Delta,\pi}-1},\ell]_{\ell^N}
            = \frac{(-1)^N}{\ell^N}\Tr_{\Q_{\ell}(\zeta)/\Q_{\ell}} \left({\pi'}^{M}\right)\ \pmod{\ell^N}.
		\end{align}

	This can be calculated by \Cref{weisman} as follows: We first note $
\Tr_{\Q_{\ell}(\zeta)/\Q_{\ell}}(\zeta^n) =(\ell^N - \ell^{N-1})\1_{\ell^N \mid n}-\ell^{N-1}\1_{\ell^{N-1}\| n}=\ell^N\1_{\ell^N \mid n}-\ell^{N-1}\1_{\ell^{N-1}\mid n}
$ for positive integers $n$, and we have
		\begin{align}
        \Tr_{\Q_{\ell}(\zeta)/\Q_{\ell}} \left({\pi'}^{M}\right)
			&=  \sum_{0\leq i\leq M} \binom{M}{i}(-1)^i\Tr_{\Q_{\ell}(\zeta)/\Q_{\ell}}(\zeta^i)\\
			&= \ell^{N-1}\left(
			\ell \sum_{\substack{0\leq i\leq M\\ \ell^{N}\mid i}} \binom{M}{i}(-1)^i-\sum_{\substack{0\leq i\leq M\\ \ell^{N-1} \mid i}} \binom{M}{i}(-1)^i\right)\\
            &\equiv \ell^N\cdot (-\ell)^{N-1}\ \pmod{\ell^{2N}}.
		\end{align}
        Therefore, $[1+\pi^{f_{\Delta,\pi}-1},\ell]_{\ell^N}=-2[1+{\pi'}^{f_{\Delta,\pi}-1},\ell]_{\ell^N}
        =2\ell^{N-1} \pmod{\ell^N}.$	

\item \textbf{Calculation of $(1+\pi^{f_{\Delta,\pi}-1},1+c)_{\ell^N}$.}
    \begin{itemize}[leftmargin=0pt]
        \item \textbf{Case (A) $b\neq 0$.}

Since $w' = w(1+c) = \ord_{\ell}(c) \geq 1$, we have
$$f(1+c) =
\begin{cases}
	2\ell^{N-w'} & (1\leq w'\leq N) \\
	0 & (w'>N).
\end{cases}$$

Thus, we get
$$f_{\Delta,\pi}-1 = f(\ell)-1 \geq f(1+c)$$
since $b\neq 0$ implying $w=0$.
Therefore, $(1+\pi^{f_{\Delta,\pi}-1},1+c)_{\ell^N} = 1$.

\item \textbf{Case (B) $b = 0$.}

In this case, we know that $f_{\Delta,\pi} = 2\ell^{N-\ord_{\ell}(c)}$
since $w = \ord_{\ell}(c) \geq 1$.
Thus, the claim follows by \Cref{1+c} and \Cref{thm:conjecture-6-1}.
    \end{itemize}
    \end{itemize}
\end{proof}

The above result allows us to compute the relative root number $W(\phi_{\delta,\pi},\eta_{\pi_F})$, and consequently $W(\phi_{\delta,\pi})$ together with the following lemma:
\begin{lemmaeigo}
	$$\chi_{\delta}|_{\A_F\x}=1.$$
\end{lemmaeigo}
\begin{proof}
	This follows in exactly the same way as \cite[Lemma 3.2]{shu2021root}.
\end{proof}

\begin{propositioneigo}\label{relative}

Suppose $b=\ord_{\ell}(\delta)$ and $r^rs^s(\ell^N-t)^t \delta^{r+s}=\epsilon'\ell^{b'}(1+c')\ (\epsilon'\in\mu_{\ell-1}(\Q_{\ell}),\ b'\in\Z,\ c'\in\ell\Z_{\ell})$ (since $\ell \nmid rst$, we have $b'=(r+s)b$ ).
	Then, the conductor exponent $f'$ of $\phi_{\delta,\pi}$ is
	$$f'=
\begin{dcases}
	{\ell^{N-1}(\ell+1)}\ &(b\neq 0)\\
	{2\ell^{N-\ord_{\ell}(c')}}\ &(b=0,\ 1\leq \ord_{\ell}(c')\leq N)\\
	1\ &(b=0,\ \ord_{\ell}(c')>N).
\end{dcases}$$
Then, we have
$$W(\phi_{\delta,\pi},\eta_{\pi_F})=
\begin{dcases}
	-\leg{(-1)^{N+1} rstb(r+s)}{\ell}\ &(b\neq 0)\\
	-\leg{(-1)^{N+1} rst c'_0}{\ell}\ &(b=0,\ 1\leq \ord_{\ell}(c')\leq N)\\
	\leg{2}{\ell}\ &(b=0,\ \ord_{\ell}(c')>N)
\end{dcases}$$
where $c'_0:=\dfrac{c'}{\ell^{\ord_{\ell}(c')}}$.
\end{propositioneigo}

\begin{proof}
	First, we compute $f'$. Using the prime element $\pi_{\rm{CM}}$ of $\Q_{\ell}(\zeta_{\ell})$ that satisfies $\pi_{\rm{CM}}^{\ell-1}=-\ell,\ \pi_{\rm{CM}}/(1-\zeta^{\ell^{N-1}})\equiv 1\pmod{\pi_{\rm{CM}}}$ as in \cite[Theorem 5.3]{coleman1988stable}, we have
	$$\phi_{\delta,\pi}(x)=(x,\delta)_{\ell^N}^{r+s} j_{r,s,t,\pi}(x)
	=(x,r^rs^s(\ell^N-t)^t \delta^{r+s})_{\ell^N}(x,\pi_{\rm{CM}})_2\ (x\in \mc{O}_{\Q_{\ell}(\zeta)}\x).$$
	Here, we note again that the Hilbert symbol in \cite[p. 43]{coleman1988stable} is the inverse of the usual definition and the definition of Jacobi sum is also different.
	Write $\pi_{\rm{CM}}=u\pi^{\ell^{N-1}}$ for some unit $u\in \mc{O}_{\Q_{\ell}(\zeta)}\x$, then we find $(x,\pi_{\rm{CM}})_2=(x,u)_2(x,\pi)_2^{\ell^{N-1}}=(x,\pi)_2$ by the formula for tamely ramified Hilbert symbols \cite[Chapter V, Proposition 3.4]{neukirch2013algebraic}. Moreover, the conductor exponent of $(x,\pi)_2$ is $1$. On the other hand,
	$$w'=w(r^rs^s(\ell^N-t)^t \delta^{r+s})
	=\begin{cases}
		\ord_{\ell}(c') &(b=0)\\
		0 &(b\neq 0)
	\end{cases}$$
    since $\ord_{\ell}(\delta) = 0$ or $\ell \nmid \ord_{\ell}(\delta)$ and $b'=(r+s)b$.
	From \Cref{conductor}, when $w'>N$ or $w'=N\neq \ord_{\ell}(c')$ (in fact, the latter case does not occur), we have $f'=1$. In all other cases, it is equal to the conductor of $(x,r^rs^s(\ell^N-t)^t \delta^{r+s})_{\ell^N}$, which can be computed in \Cref{conductor}.

	Next, we compute $W(\phi_{\delta,\pi},\eta_{\pi_F})$ using Rohrlich's formula \cite[Proposition 2]{rohrlich1992root}.

	If $f'=1$, we have $$W(\phi_{\delta,\pi},\eta_{\pi_F})=\leg{2}{\ell}.$$
	For $f'>1$,
	$$W(\phi_{\delta,\pi},\eta_{\pi_F})=\leg{-2l}{\ell}\phi_{\delta,\pi}(\pi)^{f'-1}i^u,$$
	where $u=\begin{dcases}
		0\ (\ell\equiv 1\pmod{4})\\
		1\ (\ell\equiv 3\pmod{4})
	\end{dcases}$

	and $l$ is the class determined by $\phi_{\delta,\pi}(1+\pi^{f'-1})=\zeta_{\ell}^l$.
        Since
	$$\phi_{\delta,\pi}(1+\pi^{f'-1})
	=(1+\pi^{f'-1},r^rs^s(\ell^N-t)^t \delta^{r+s})_{\ell^N},$$
	and since $b'=(r+s)b$, the value of $l$ is given by \Cref{value of hilbert symbol} as
	$$l
	=\begin{dcases}
		2b(r+s) \pmod{\ell} &(b\neq 0)\\
		-2c'_0 \pmod{\ell} &(b=0,\ 1\leq \ord_{\ell}(c')\leq N)
	\end{dcases}$$
For simplicity, we temporarily write $d:=\dfrac{\ell^{N-1}(\ell-1)}{2}$. By the above lemma and \cite[Proposition 4]{rohrlich1992root}, we have
\begin{align}
	\phi_{\delta,\pi}(\pi)
=\phi_{\pi}(\pi)
=&\prod_{V\mid \infty} \phi_{V}(\pi)\inv\\
=&\prod_{h\in\Phi\inv} \frac{\pi^{\sigma_h}}{\abs{\pi^{\sigma_h}}}\\
=&i^{d}(-1)^{d}\cdot\left\{-\leg{-2rst}{\ell}\right\}\inv\\
=&-i^{d}\leg{2rst}{\ell}.
\end{align}
Therefore, since $i^{u-d}=\leg{(-1)^{N}2}{\ell}$, we have
\begin{align}
    W(\phi_{\delta,\pi},\eta_{\pi_F})
=&\leg{-2l}{\ell}\left\{-i^{d(f'-1)}\leg{2rst}{\ell}\right\}i^u\\
=&-\leg{(-1)^{N+1} 2rstl}{\ell}i^{df'}.
\end{align}
This concludes the proof along with $\leg{-1}{\ell}=i^{\ell-1}$.
\end{proof}

\section{Global root number $W(\phi_{\delta}^{(i)})$ and global conductor $N_{r,s,t,\delta}$}
We can eventually describe the global root number.
\begin{theoremeigo}
    Let $r, s, t > 0$ be integers such that $r + s + t = \ell^N$ and $\ell \nmid rst$. Suppose $$v_{\ell}=\ord_{\ell}\left\{\left(\frac{r^rs^s(\ell^N-t)^t \delta^{r+s}}{\ell^{(r+s)\ord_{\ell}(\delta)}}\right)^{\ell-1}-1\right\}$$
    and when $\ell\nmid \delta$,
    $$c_0=-\frac{\left(r^rs^s(\ell^N-t)^t \delta^{r+s}\right)^{\ell-1}-1}{\ell^{v_{\ell}}}.$$
Then the global root number of $\phi_{\delta}^{(N)}$ is given by
$$W(\phi_{\delta}^{(N)})=\dprod_{\text{$p\le \infty$}}W_p(\phi_{\delta}^{(N)})$$
where for primes $p \neq \ell$,
$$W_p(\phi_{\delta}^{(N)})
=\begin{dcases}
	i^{-\frac{\ell^{N-1}(\ell-1)}{2}}\ &(p=\infty)\\
	\leg{p}{\ell}\ &(p\mid \delta)\\
	1\ &(p\nmid \delta)
\end{dcases} $$
and,
$$W_{\ell}(\phi_{\delta}^{(N)})=
\begin{dcases}
	-\leg{(-1)^{N+1}rst\ord_{\ell}(\delta)(r+s)}{\ell}i^{\frac{\ell^{N-1}(\ell-1)}{2}}\ &(\ell\mid\delta)\\
	-\leg{(-1)^{N+1} rstc_0}{\ell}i^{\frac{\ell^{N-1}(\ell-1)}{2}}\ &(\ell\nmid\delta,\ 1\leq v_{\ell}\leq N)\\
	\leg{2}{\ell}i^{\frac{\ell^{N-1}(\ell-1)}{2}}\ &(\ell\nmid\delta,\ v_{\ell}>N).
\end{dcases}$$
\end{theoremeigo}
\begin{proof}
    Write $r^r s^s (\ell^N-t)^t \delta^{r+s} = \epsilon'_N \ell^{b'_N}(1 + c_N')\ (\epsilon'_N \in \mu_{\ell-1}(\Q_{\ell}), b'_N \in \Z, c_N' \in \ell \Z_{\ell})$, then $v_{\ell}=\ord_{\ell}(c_N')$ and $c_N'/\ell^{v_{\ell}}\equiv c_0 \pmod{\ell}$.
    Combining this with the result of \Cref{eta}, \Cref{value of hilbert symbol}, \Cref{relative} proves the theorem.
\end{proof}

The global conductor $N_{r,s,t,\delta}$ is given by the following theorem.

\begin{theoremeigo}\label{global conductor}
    The global conductor is
    \[
        N_{r,s,t,\delta}
        =
        \ell^{\ell^{N-1}(N\ell-N-1)+f_{\ell}}
        \prod_{\substack{p\mid\delta\\p\neq\ell}}
        p^{\ell^{N-1}(\ell-1)},
    \]
    where $f_{\ell}$ is the conductor exponent defined in
    \Cref{relative}.
\end{theoremeigo}

\begin{proof}
    The conductor of $\phi_{\delta}$ is
    \[
        \mf{f}_{\phi_{\delta}}
        =
        (\pi)^{f_{\ell}}
        \prod_{\substack{p\mid\delta\\p\neq\ell}}
        \prod_{\mf{p}\mid p}\mf{p},
    \]
    where the inner product is taken over all prime ideals $\mf{p}$
    of $K$ lying above $p$.

    Indeed, the conductor exponent at the unique prime $(\pi)$ above
    $\ell$ is $f_{\ell}$ by \Cref{relative}. On the other hand, for
    every finite prime $\mf{p}\nmid\ell$, the same calculation as in
    \cite[Lemma 3.10 (2)]{shu2021root} gives
    \[
        f(\phi_{\delta,\mf{p}})
        =
        \begin{cases}
            1 & (\mf{p}\mid\delta)\\
            0 & (\mf{p}\nmid\delta).
        \end{cases}
    \]
    Since $\mathcal N((\pi))=\ell$ and every rational prime $p\ne\ell$ is unramified in $K/\Q$, we have $pO_K=\prod_{\mathfrak p\mid p}\mathfrak p$.
    Hence,
    \[
        \mc{N}
        \left(
            \prod_{\mf{p}\mid p}\mf{p}
        \right)
        =
        \mc{N}(p\mc{O}_{K})
        =
        p^{[K:\Q]}
        =
        p^{\ell^{N-1}(\ell-1)}.
    \]
    Therefore, the theorem follows from $\abs{D_K}=\ell^{\ell^{N-1}(N\ell-N-1)}$
    and $N_{r,s,t,\delta}=\abs{D_K}\,
        \mc{N}(\mf{f}_{\phi_{\delta}})$.        
\end{proof}

\begin{exampleeigo}
    Consider the case when $N = 2,\ \ell=3$. Up to the equivalence described in \cite[(1),(2)]{coleman1989torsion}, the possible pairs $(r,s)$ are $(r,s)=(1,1),(1,4)$ and $(2,2)$ .

    The tables of the global root number $W(\phi_{\delta}^{(N)})$ and the global conductor $N_{r,s,t,\delta}$ in this case are as follows.

    \begin{table}[htbp]
\centering
\begin{tabular}{|c|c|c|c|c|c|c|c|c|}
\hline
$\delta$
& $1$ & $2$ & $3$ & $4$ & $5$ & $6$ & $7$ & $8$ \\
\hline
$\ord_{\ell}(c'_N)$
& $1$ & $3$ & $1$ & $1$ & $1$ & $3$ & $2$ & $1$ \\
\hline
$W(\phi_{\delta}^{(N)})$
& $-1$ & $1$ & $-1$ & $-1$ & $-1$ & $1$ & $-1$ & $1$ \\
\hline
$N_{r,s,t,\delta}$
& $3^{15}$
& $2^{6}\cdot 3^{10}$
& $3^{21}$
& $2^{6}\cdot 3^{15}$
& $3^{15}\cdot 5^{6}$
& $2^{6}\cdot 3^{21}$
& $3^{11}\cdot 7^{6}$
& $2^{6}\cdot 3^{15}$ \\
\hline
\end{tabular}
\caption{The case when $(r,s)=(1,1)$.}
\end{table}

\begin{table}[htbp]
\centering
\begin{tabular}{|c|c|c|c|c|c|c|c|c|}
\hline
$\delta$
& $1$ & $2$ & $3$ & $4$ & $5$ & $6$ & $7$ & $8$ \\
\hline
$\ord_{\ell}(c'_N)$
& $1$ & $2$ & $1$ & $1$ & $1$ & $2$ & $3$ & $1$ \\
\hline
$W(\phi_{\delta}^{(N)})$
& $-1$ & $-1$ & $-1$ & $-1$ & $-1$ & $1$ & $-1$ & $1$ \\
\hline
$N_{r,s,t,\delta}$
& $3^{15}$
& $2^{6}\cdot 3^{11}$
& $3^{21}$
& $2^{6}\cdot 3^{15}$
& $3^{15}\cdot 5^{6}$
& $2^{6}\cdot 3^{21}$
& $3^{10}\cdot 7^{6}$
& $2^{6}\cdot 3^{15}$ \\
\hline
\end{tabular}
\caption{The case when $(r,s)=(1,4)$.}
\end{table}

\begin{table}[htbp]
\centering
\begin{tabular}{|c|c|c|c|c|c|c|c|c|}
\hline
$\delta$
& $1$ & $2$ & $3$ & $4$ & $5$ & $6$ & $7$ & $8$ \\
\hline
$\ord_{\ell}(c'_N)$
& $1$ & $3$ & $1$ & $1$ & $1$ & $3$ & $2$ & $1$ \\
\hline
$W(\phi_{\delta}^{(N)})$
& $-1$ & $1$ & $-1$ & $-1$ & $-1$ & $1$ & $-1$ & $1$ \\
\hline
$N_{r,s,t,\delta}$
& $3^{15}$
& $2^{6}\cdot 3^{10}$
& $3^{21}$
& $2^{6}\cdot 3^{15}$
& $3^{15}\cdot 5^{6}$
& $2^{6}\cdot 3^{21}$
& $3^{11}\cdot 7^{6}$
& $2^{6}\cdot 3^{15}$ \\
\hline
\end{tabular}
\caption{The case when $(r,s)=(2,2)$.}
\end{table}

\end{exampleeigo}

\bibliographystyle{abbrv}
\bibliography{test}

\end{document}